\documentclass[reqno,a4paper]{amsart}
\usepackage{amssymb}
\usepackage{amsmath}
\newcommand{\angles}[1]{\langle #1 \rangle}
\newcommand{\half}{\frac{1}{2}}
\newcommand{\abs}[1]{\vert #1 \vert}
\newcommand{\norm}[1]{\left\Vert #1 \right\Vert}
\newcommand{\R}{\mathbb{R}}

\begin{document} 
\newtheorem{prop}{Proposition}[section]
\newtheorem{Def}{Definition}[section]
\newtheorem{theorem}{Theorem}[section]
\newtheorem{lemma}{Lemma}[section]
 \newtheorem{Cor}{Corollary}[section]

\title[LWP for Yang-Mills-Dirac]{\bf Local well-posedness of the coupled Yang-Mills and Dirac system for low regularity data}
\author[Hartmut Pecher]{
{\bf Hartmut Pecher}\\
Fakult\"at f\"ur  Mathematik und Naturwissenschaften\\
Bergische Universit\"at Wuppertal\\
Gau{\ss}str.  20\\
42119 Wuppertal\\
Germany\\
e-mail {\tt pecher@uni-wuppertal.de}}
\date{}

\begin{abstract}
We consider the classical Yang-Mills system coupled with a Dirac equation in 3+1 dimensions. Using that most of the nonlinear terms fulfill a null condition we prove local well-posedness for data with minimal regularity assumptions. This problem for smooth data was solved forty years ago by Y. Choquet-Bruhat and D. Christodoulou. Our result generalizes a similar result for the Yang-Mills equation by S. Selberg and A. Tesfahun.  
\end{abstract}
\maketitle
\renewcommand{\thefootnote}{\fnsymbol{footnote}}
\footnotetext{\hspace{-1.5em}{\it 2020 Mathematics Subject Classification:} 
35Q40, 35L70 }
\normalsize 
\setcounter{section}{0}

\section{Introduction and the main theorem}

\noindent 
Let $\mathcal{G}$ be the Lie group $SU(n,\mathbb{C})$ (the group of unitary matrices of determinant 1) and $g$ its Lie algebra $su(n,\mathbb{C})$ (the algebra of trace-free skew hermitian matrices) with Lie bracket $[X,Y] = XY-YX$ (the matrix commutator). 
For given  $A_{\alpha}: \mathbb{R}^{1+3} \rightarrow g $ we define the curvature $F=F[A]$ by
\begin{equation}
\label{curv}
 F_{\mu \nu} = \partial_{\mu} A_{\nu} - \partial_{\nu} A_{\mu} + [A_{\mu},A_{\nu}] \, , 
\end{equation}
where $\mu,\nu \in \{0,1,2,3\}$ and $D_{\mu} = \partial_{\mu} + [A_{\mu}, \cdot \,]$ .

Then the Yang-Mills system is given by
\begin{equation}
\label{0}
D^{\mu} F_{\mu \nu}  = 0
\end{equation}
in Minkowski space $\mathbb{R}^{1+3} = \mathbb{R}_t \times \mathbb{R}^3_x$ , with metric $diag(-1,1,1,1)$. Greek indices run over $\{0,1,2,3\}$, Latin indices over $\{1,2,3\}$, and the usual summation convention is used.  
We use the notation $\partial_{\mu} = \frac{\partial}{\partial x_{\mu}}$, where we write $(x^0,x^1,x^2,x^3)=(t,x^1,x^2,x^3)$ and also $\partial_0 = \partial_t$.

This system is coupled with a Dirac spinor field $\psi: \R^{1+3} \to \mathbb{C}^4$ . Let $T_a$ be the set of generators
 of $SU(n,\mathbb{C})$ and $A_{\mu} = A^a_{\mu} T_a$ , $F_{\mu \nu} = F^a_{\mu \nu} T_a$ , $[T^\lambda,T^b]_a =: f^{ab \lambda}$ .

For the following considerations and also for the physical background we refer to the monograph by Matthew D. Schwartz \cite{Sz} . We also refer to the pioneering work for the Yang-Mills, Higgs and spinor field equations by Y. Choquet-Bruhat and D.  Christodoulou \cite{CC} , and G. Schwarz and J. Sniatycki \cite{SS}.

The kinetic Lagrangian with $N$ Dirac fermions and the Yang-Mills Lagrangian are given by
$\mathcal{L} = \sum_{j=1}^N \bar{\psi}_j  (i \gamma^{\mu} \partial_{\mu}-m)\psi_j $ and $\mathcal{L}_{YM} = - \frac{1}{4} (F^a_{\mu \nu})^2$ , respectively. Here $\bar{\psi} = \psi^{\dagger} \gamma^0$ , where $\psi^{\dagger}$ is the complex conjugate transpose of $\psi$ .

Here $\gamma^{\mu}$ are the (4x4) Dirac matrices given by
$ \, \, \gamma^0 = \left( \begin{array}{cc}
I & 0  \\ 
0 & -I  \end{array} \right)\, \,$ 
 , $\, \,  \gamma^j = \left( \begin{array}{cc}
0 & \sigma^j  \\
-\sigma^j & 0  \end{array} \right) \, \, $ , where  $\, \, \sigma^1 = \left( \begin{array}{cc}
0 & 1  \\
1 & 0  \end{array} \right)$ ,
$ \sigma^2 = \left( \begin{array}{cc}
0 & -i  \\
i & 0  \end{array} \right)$ ,
$ \sigma^3 = \left( \begin{array}{cc}
1 & 0  \\
0 & -1  \end{array} \right)$ .
Then we consider the following Lagrangian for the (minimally)  coupled system
\begin{align*}
 \mathcal{L}&=- \frac{1}{4} (F^a_{\mu \nu})^2 + \sum_{i,j=1}^N \bar{\psi}_i (\delta_{ij} i \gamma^{\mu} \partial_{\mu}+\gamma^{\mu}A^a_{\mu} T^a_{ij}     -m \delta_{ij})\psi_j \\
& =-\frac{1}{4}(\partial_{\mu} A^a_{\nu}-\partial_{\nu} A^a_{\mu} + f^{abc} A^b_{\mu} A^c_{\nu})^2 +\sum_{i,j=1}^N \bar{\psi}_i (\delta_{ij} i \gamma^{\mu} \partial_{\mu}+\gamma^{\mu}A^a_{\mu} T^a_{ij}     -m \delta_{ij})\psi_j \, . 
\end{align*}
Here $T^a_{ij} \in \mathbb{C}$ are the entries of the matrix $T^a$ . 

The corresponding equations of motion are given by the following coupled Yang-Mills-Dirac system (YMD)
\begin{align*} 
 \partial^{\mu} F^a_{\mu \nu} + f^{abc} A^b_{\mu} F^c_{\mu \nu} &= - \langle \psi_i,\gamma_0 \gamma_{\nu} T^a_{ij} \psi_j \rangle \\
(i \gamma^{\mu} \partial_{\mu} -m) \psi_i & = - A^a_{\mu} \gamma^{\mu} T^a_{ij} \psi_j \, .
\end{align*}
Using $D^{\mu} F_{\mu \nu} = \partial^{\mu} F_{\mu \nu} + [A^{\mu},F_{\mu \nu}]$ and 
$$ [A^{\mu},F_{\mu \nu}]_a = [A^{\lambda}_{\mu} T_{\lambda},F^b_{\mu \nu} T_b]_a = A^{\lambda}_{\mu} F^b_{\mu \nu} [T_{\lambda},T_b]_a = A^b_{\mu} F^b_{\mu \nu} f_{ab\lambda} $$
we obtain the following system which we intend to treat:
\begin{align}
\label{0.1}
D^{\mu} F_{\mu \nu} & = - \langle \psi^i,\alpha_{\nu} T^a_{ij} \psi^j \rangle T_a \, , \\
\label{0.2}
i \alpha^{\mu} \partial_{\mu} \psi_i & = -A^a_{\mu} \alpha^{\mu} T^a_{ij} \psi_j \, ,
\end{align}
if we choose $m=0$ just for simplicity and define the matrices $\alpha^{\mu} = \gamma^0 \gamma^{\mu}$ , so that $\, \,\alpha^0 = I_{4x4}$ and $ \alpha^j = \left( \begin{array}{cc}
0 & \sigma^j  \\
\sigma^j & 0  \end{array} \right)$.
$\alpha^{\mu}$ are hermitian matrices with $(\alpha^{\mu})^2 = I_{4x4}$ , $\alpha^j \alpha^k + \alpha^k \alpha^j = 0$ for $j \neq k$ .

Following \cite{AFS1} and \cite{HO} in order to  rewrite the Dirac equation we define the projections
$$\Pi(\xi) := \half(I_{4x4} + \frac{\xi_j \alpha^j}{|\xi|})$$  and $\Pi_{\pm}(\xi):= \Pi(\pm \xi)$ , so that $\Pi_{\pm}(\xi)^2 = \Pi_{\pm}(\xi)$ , $\Pi_+(\xi) \Pi_-(\xi) =0 $ , $\Pi_+ (\xi) + \Pi_-(\xi) = I_{4x4}$ , $\Pi_{\pm}(\xi) = \Pi_{\mp}(-\xi) $ .
We obtain 
\begin{equation}
\label{2.6'}
\alpha^j \Pi(\xi) = \Pi(\pm \xi) \alpha^j + \frac{\xi_j}{|\xi|} I_{4x4} \, .
\end{equation}
Using the notation $\Pi_{\pm} = \Pi_{\pm}(\frac{\nabla}{i})$ we obtain
\begin{equation}
\label{2.8}
 -i\alpha^j \partial_j = |\nabla|\Pi_+ - |\nabla|\Pi_- \, , 
\end{equation}
where $|\nabla|$ has symbol $|\xi| $ . Moreover defining the modified Riesz transform by $R^j_{\pm} = \mp(\frac{\partial_j}{i|\nabla|}) $ with symbol $ \mp \frac{\xi_j}{|\xi|}$ and $R^0_{\pm} = -1$ the identity (\ref{2.6'}) implies
\begin{equation}
\label{2.7}
\alpha^j \Pi_{\pm} = (\alpha^j \Pi_{\pm})\Pi_{\pm} = \Pi_{\mp} \alpha^j \Pi_{\pm}- R^j_{\pm} \Pi_{\pm} \, , \, \alpha^0 \Pi_{\pm}= \Pi_{\pm} = \Pi_{\mp} \alpha^0 \Pi{\pm} - R^0_{\pm} \Pi_{\pm} \, .
\end{equation}
If we define $\psi_{i,\pm} = \Pi_{\pm} \psi_i$ we obtain by applying the projection $\Pi_{\pm}$ and (\ref{2.8}) the Dirac type equation in the form
\begin{equation}
\label{0.3}
(i \partial_t \pm |\nabla|)\psi_{i,\pm} = \Pi_{\pm}(A^a_{\mu} \alpha^{\mu} T^a_{ij} \psi^j) = : H_{i,\pm}(A,\psi) \, .
\end{equation}

The Yang-Mills equation (\ref{0.1}) may be written as
$$\square  A_{\nu} = \partial_{\nu} \partial^{\mu} A_{\mu}-[\partial^{\mu} A_{\mu},A_{\nu}] - [A_{\mu},\partial^{\mu} A_{\nu}] - [A^{\mu},F_{\mu \nu}] - \langle \psi_i,\alpha_{\nu} T^a_{ij} \psi^j \rangle T_a \, . $$

From now on we impose the Lorenz gauge condition
$$\partial_{\mu} A^{\mu} = 0 \, . $$
This implies the wave equation
\begin{align}
\nonumber
\square  A_{\nu} &= - 2[A^\mu,\partial_\mu A_\beta] + [A^\mu,\partial_\nu A_\mu] - 
  [A^\mu, [A_\mu,A_\nu]]- \langle \psi^i,\alpha_{\nu} T^a_{ij} \psi^j \rangle T_a \\
\label{0.4}
&=: K_{\nu}(A,F) + J_{\nu}(\psi) \, .
\end{align}

In order to derive a wave equation for $F_{\beta \gamma}$ we follow the arguments of \cite{ST}. Applying $D_{\mu}$ to the Bianchi identity
$$ D_{\mu} F_{\beta \gamma} + D_{\mu} F_{\gamma \mu} + D_{\gamma} F_{\mu \beta}= 0$$
and using the relations
$$D^{\mu} D_{\beta} F_{\gamma \mu} = D_{\beta} D^{\mu} F_{\gamma \mu} + [F^{\mu \beta},F_{\gamma \mu}] $$
and
$$D^{\mu} D_{\gamma} F_{\mu \beta} = D_{\gamma} D^{\mu} F_{\mu \beta} + [F^{\mu \gamma},F_{\mu \beta}] $$
we obtain
$$ D^{\mu} D_{\mu} F_{\beta \gamma} + D_{\beta} D^{\mu} F_{\gamma \mu} + [F^{\mu \beta},F_{\gamma \mu}] + D_{\gamma} D^{\mu} F_{\mu \beta} + [F^{\mu \gamma},F_{\mu \beta}] = 0 \ . $$
But by the Yang-Mills equations we obtain
$$ D^{\mu} F_{\gamma \mu}= - D^{\mu} F_{\mu \gamma} = - J_{\gamma}(\psi) \, \Rightarrow D_{\beta} D^{\mu} F_{\gamma \mu} = - D_{\beta} J_{\gamma}(\psi) $$
and
$$ D^{\mu} F_{\mu \beta}=  J_{\beta}(\psi) \, \Rightarrow D_{\gamma} D^{\mu} F_{\mu \beta} = - D_{\gamma} J_{\beta}(\psi) \, ,$$
which implies
\begin{equation}
\label{*0}
D^{\mu} D_{\mu} F_{\beta \gamma} - D_{\beta} J_{\gamma}(\psi) + D_{\gamma} J_{\beta}(\psi) + 2 [F^{\mu \beta},F_{\gamma \mu}] = 0  \, . 
\end{equation}
Now we use the identity (cf. \cite{ST})
$$ D^{\mu} D_{\mu} X = \square X + [A^{\mu},\partial_{\mu} X ] + \partial^{\mu} [A_{\mu},X] +  [A^{\mu},[A_{\mu},X]] $$ and obtain the wave equation
\begin{align}
\label{0.5}
\nonumber
\square F_{\beta \gamma}  =& (-[A^{\alpha},\partial_{\alpha} F_{\beta \gamma}] - \partial^{\alpha} [A_{\alpha},F_{\beta \gamma}] - [A^{\alpha},[A_{\alpha},F_{\beta \gamma}]] - 2 [F^{\alpha \beta},F_{\gamma \alpha}]) \\
\nonumber
& - D_{\beta} J_{\gamma}(\psi) +  D_{\gamma} J_{\beta}(\psi) \\
 =:& \, L_{\beta \gamma} (A,F) + I_{\beta \gamma}(A,\psi) \, .
\end{align}

Expanding the second and fourth terms in \eqref{0.5}, and also imposing the Lorenz gauge, yields

\begin{equation}\label{YMF2}
\begin{split}
       \square F_{\beta\gamma}&= - 2[A^\alpha,\partial_\alpha F_{\beta\gamma}]
      + 2[\partial_\gamma A^\alpha, \partial_\alpha A_\beta]
      - 2[\partial_\beta A^\alpha, \partial_\alpha A_\gamma]
      \\
      &\quad + 2[\partial^\alpha A_\beta , \partial_\alpha A_\gamma]
                 + 2[\partial_\beta A^\alpha, \partial_\gamma A_\alpha] - [A^\alpha,[A_\alpha,F_{\beta\gamma}]] 
                 \\
                 &\quad  +2[F_{\alpha\beta},[A^\alpha,A_\gamma]]- 2[F_{\alpha\gamma},[A^\alpha,A_\beta]]
                      - 2[[A^\alpha,A_\beta],[A_\alpha,A_\gamma]] \\
 &\quad - D_{\beta} J_{\gamma}(\psi) +  D_{\gamma} J_{\beta}(\psi) \\
&  =: \, L_{\beta \gamma} (A,F) + I_{\beta \gamma}(A,\psi) \, .
                         \end{split}    
\end{equation}

This implies the following equivalent system which we consider from now on:
\begin{align}
\label{1.1a}
(i \partial_t \pm | \nabla |  ) \psi_{i,\pm} & = H_{i,\pm}(A,\psi)\, , \\
\label{1.1b}
\square A_{\nu} & =  K_{\nu}(A,F) + J_{\nu}(\psi) \, , \\
\label{1.1c}
\square F_{\mu \nu} & = L_{\mu \nu}(A,F) + I_{\mu \nu}(A,\psi)\, .
\end{align}

We want to solve the system (\ref{1.1a}),(\ref{1.1b}),(\ref{1.1c}) simultaneously for $A$ , $F$ and $\psi_{\pm}$ .
So to pose the Cauchy problem for this system, we consider initial data for $(A,F,\psi)$ at $t=0$:
\begin{equation}\label{Data}
\begin{split}
&A_{\nu}(0) = a_0^{\nu}, \, (\partial_t A_{\nu})(0) =  a_1^{\nu},
        \,
    F_{\mu \nu}(0) =f_0^{\nu}, \, (\partial_t F_{\mu \nu})(0) =  f_1^{\mu \nu}, \\ & \, \psi_{i,\pm}(0) = \psi_{i,\pm}^0 =  \Pi_{\pm} \psi_0^i .
\end{split}
 \end{equation}

In fact, the initial data for $F$ can be determined from $(a_0^{\nu}, a_1^{\nu}, \psi_0^i)$ as follows:
\begin{align}
\label{f1}
  f_0^{ij} &= \partial_i a_0^j - \partial_j a_0^i + [a_0^i,a_0^j],
  \\
\label{f2}
  f_0^{0i} &=  a_1^i - \partial_i a_0^0 + [a_0^0,a_0^i],
\\ \label{f3}
   f_1^{ij} &= \partial_i  a_1^j - \partial_j  a_1^i + [ a_1^i,a_0^j]+[ a_0^i, a_1^j],
  \\ \label{f4}
 f_1^{0j} &= \partial^k f_0^{kj} +[a_0^{\mu}, f_0^{\mu j}] + \langle \psi_0^i,  \alpha_j T^a_{ij} \psi_0^j \rangle T_a \, ,
\end{align}
where the first three expressions come from \eqref{curv} whereas 
the last one comes from (\ref{0.1}) with $\nu=j$.

Note that the Lorenz gauge condition $\partial^\alpha A_\alpha=0$ and (\ref{0.2}) with $\nu=0$ impose the constraints 
\begin{equation}\label{Const}
 a_1^0= \partial^i a_0^i,
\quad
   \partial^k f_0^{k0} + [a_0^k, f_0^{k0}] = - \langle \psi_0^i,T^a_{ij} \psi_0^j \rangle T_a \, .
\end{equation}

Our main theorem reads as follows:
\begin{theorem}
\label{Theorem1.1}
Assume that $s$ , $r$ and $l$ satisfy the following conditions:
\begin{align*}
&s > \frac{3}{4} \, ,  \quad r >-\frac{1}{8} \, , \quad l > \frac{3}{8} \\
& s \ge l \ge r \quad , \quad
2l-s > 0 \quad , \quad 2r-s > -1  \, ,\\
& 2l-r >  \half \quad , \quad 3l-2r > \frac{3}{4} \quad , \quad 2s-r > \frac{3}{2}  \quad , \quad 3s-2r > \frac{7}{4} \, . \, 
\end{align*}
Given initial data $a^{\nu}_0 \in H^s$ , $a^{\nu}_1 \in H^{s-1}$ , $f^{\mu \nu}_0 \in H^r$ , $f^{\mu \nu}_1 \in H^{r-1}$ , $\psi_{i,\pm}^0 \in H^l$, which fulfill (\ref{f1})-(\ref{Const}), there exists a time $T > 0$ , depending on the norms of the data, such that the Cauchy problem  (\ref{Data}) for  the Yang-Mills-Dirac system (\ref{0.2}),(\ref{0.4}),(\ref{0.5}) in Lorenz gauge has a unique solution $$A \in  F^s_T \quad , \quad F \in G^r_T \quad , \quad
  \psi \in X^{l,\frac{3}{4}+}_+[0,T] + X^{l,\frac{3}{4}+}_-[0,T]$$  (these spaces are defined in Def. \ref{Def.1.2}). It has the regularity
$$ A \in C^0([0,T],H^s) \cap C^1([0,T],H^{s-1}) \quad , \quad F \in C^0([0,T],H^r) \cap C^1([0,T],H^{r-1}) \, , $$
$$ \psi \in C^0([0,T],H^l) \, . $$
The solution depends continously on the data and higher regularity persists.
\end{theorem}
{\bf Remark 1:} The most natural relation between $s$ and $r$ is $r=s-1$ . In this case the condition $2r-s >-1$ would force $s > 1$ , which would exclude the most interesting range $ \frac{3}{4} < s \le 1$ .\\
{\bf Remark 2:} The assumption $s >\frac{3}{4}$ is necessary for one of the estimates, and therefore $2r > s-1 > - \frac{1}{4}$, thus $r > -\frac{1}{8}$. Moreover $2l-s > 0$ $\Leftrightarrow$ $l> \frac{s}{2}  > \frac{3}{8}$. One easily checks that the choice $s= \frac{3}{4} + \epsilon$ , $ r = - \frac{1}{8}+\epsilon$ , $l= \frac{3}{8}+\epsilon$  satisfies our assumptions for arbitrary $\epsilon > 0$ .

We prove local well-posedness by iterating in the $X^{s,b}$-spaces adapted to the operators $i\partial_t \pm \langle \nabla \rangle $ and $\square$ . 
\begin{Def}
\label{Def.1.2} 
$X^{s,b}_{\pm}$ is the completion of $\mathcal S(\R^{1+3})$ with respect to the norm
$$
  \norm{u}_{X^{s,b}_\pm} = \| \angles{\xi}^s \langle -\tau \pm |\xi| \rangle^b \widetilde u(\tau,\xi) \|_{L^2_{\tau,\xi}},
$$
where $\widetilde u(\tau,\xi) = \mathcal F_{t,x} u(\tau,\xi)$ is the space-time Fourier transform of $u(t,x)$.

 Let $X^{s,b}_\pm [0,T]$ denote the restriction space to the interval $[0,T]$ for $T>0$.
In addition to $X^{s,b}_\pm$, we shall also need the wave-Sobolev spaces $H^{s,b}$, defined to be the completion of $\mathcal S(\R^{1+3})$ with respect to the norm
$$
  \norm{u}_{H^{s,b}} = \norm{\angles{\xi}^s \langle |\tau|-  |\xi| \rangle^b \widetilde u(\tau,\xi)}_{L^2_{\tau,\xi}}.
$$
Moreover we define the spaces $F^s$ and $G^r$ by their norms
\begin{align*}
\|u\|_{F^s} &:= \|\Lambda_+ u\|_{H^{s-1,\frac{3}{4}+\epsilon}} \\
\|v\|_{G^r} &:= \|\Lambda_+ v\|_{H^{r-1,\half+\epsilon}}  \ ,
\end{align*}
where $\epsilon > 0$ is sufficiently small. $F^s_T$ and $G^r_T$ denotes the restriction to the time interval $[0,T]$.

 \end{Def}
We recall the fact that
 $$
 X^{s,b}_\pm [0,T]  \hookrightarrow C^0([-T,T];H^s) \quad \text{for} \ b > \half.
$$
We use the following notation:
let $\Lambda^{\alpha}$, $\Lambda_{+}^{\alpha}$ and $\Lambda_{-}^{\alpha}$
be the multipliers with symbols  $$
\langle\xi \rangle^\alpha,
\quad \langle|\tau|  + |\xi|\rangle^\alpha,
\quad \langle|\tau|-|\xi|\rangle^\alpha.
$$
Similarly let $D^{\alpha}$,
$D_{+}^{\alpha}$ and $D_{-}^{\alpha}$ be the multipliers with symbols $$
\abs{\xi}^\alpha,
\quad (\abs{\tau} + \abs{\xi})^\alpha,
\quad ||\tau|-|\xi||^\alpha,
$$
respectively.

Let us make some historical remarks. As is well-known we may impose a gauge condition. We exlusively study the Lorenz gauge $\partial^{\alpha}A_{\alpha} =0$. Other convenient gauges are the Coulomb gauge $\partial^j A_j=0$ and the temporal gauge $A_0 =0$. It is well-known that for the low regularity well-posedness problem for the Yang-Mills equation a null structure for some of the nonlinear terms plays a crucial role. This was first detected by Klainerman and Machedon \cite{KM}, who proved global well-posedness in the case of three space dimensions in temporal and in Coulomb gauge in energy space. The corresponding result in Lorenz gauge, where the Yang-Mills equation can be formulated as a system of nonlinear wave equations, was shown by Selberg and Tesfahun \cite{ST}, who discovered that also in this case some of the nonlinearities have a null structure.  Tesfahun \cite{Te} improved the local well-posedness result to data without finite energy, namely for $(A(0),(\partial_t A)(0) \in H^s \times H^{s-1}$ and $(F(0),(\partial_t F)(0) \in H^r \times H^{r-1}$ with $s > \frac{6}{7}$ and $r > -\frac{1}{14}$, by discovering an additional partial null structure. 
Local well-posedness in energy space was also shown by Oh \cite{O} using a new gauge, namely the Yang-Mills heat flow. He was also able to show that this solution can be globally extended \cite{O1}. Tao \cite{T} showed local well-posedness for small data in $H^s \times H^{s-1}$ for $ s > \frac{3}{4}$ in temporal gauge. 

The local well-posedness problem for coupled Yang-Mills-Higgs equations in Lorenz gauge for low regularity data was considered by A. Tesfahun \cite{Te1}.

In the present paper we treat the coupled Yang-Mills-Dirac equation in Lorenz gauge for space dimension $n=3$  which was considered from the physical point of view by M. D. Schwartz \cite{Sz}. Local existence for smooth initial data, uniqueness in suitable gauges under appropriate conditions on the data and global existence for small and smooth data , i.e. $(A(0),(\partial_t A)(0),F(0),(\partial_t F)(0),\psi(0)) \in H^s \times H^{s-1} \times H^{s-1} \times H^{s-2} \times H^s$ with $s \ge 2$  was proven by  Y. Choquet-Bruhat and D. Christodoulou \cite{CC}, and G. Schwarz and J. Sniatycki \cite{SS}.

Our main result (Theorem \ref{Theorem1.1}) is local well-posedness for $s > \frac{3}{4}$ , $r >-\frac{1}{8}$ and $l > \frac{3}{8}$ , where existence holds in $ A \in C^0([0,T],H^s) \cap C^1([0,T],H^{s-1}) \, , \, F \in C^0([0,T],H^r) \cap C^1([0,T],H^{r-1}) $ , $\psi \in C^0([0,T],H^l)$ and (existence and) uniqueness in a certain subspace of Bourgain-Klainerman-Machedon type $X^{s,b}$  (Theorem \ref{Theorem1.1}). Thus the assumptions on the Cauchy data are significantly weakened. For an essential part of the necessary estimates we rely on Selberg-Tesfahun \cite{ST} and Tesfahun's  result \cite{Te}, who detected   the null structure in most - unfortunately not all  - critical nonlinear terms.  Important is the convenient  atlas of bilinear estimates in wave-Sobolev spaces by \cite{AFS} . We also make use of the methods used by Huh and Oh \cite{HO} for the Chern-Simons-Dirac equation.

Finally we remark that even in the case of the pure Yang-Mills system the regularity assumptions on the Cauchy data are slightly weakened compared to Tesfahun's result \cite{Te}. For this result we have to modify the solution spaces appropriately. Here we rely on some of the results which were used by Klainerman and Selberg \cite{KS} in order to prove an almost optimal well-posedness result for (a model problem of) the Yang-Mills equation in 4+1 dimensions.

\section{Preliminaries}
  The following product estimates for wave-Sobolev spaces were proven in \cite{AFS}.
\begin{prop}
\label{Prop.1.2'} Let $n=3$ .
 For $s_0,s_1,s_2,b_0,b_1,b_2 \in {\mathbb R}$ and $u,v \in   {\mathcal S} ({\mathbb R}^{3+1})$ the estimate
$$\|uv\|_{H^{-s_0,-b_0}} \lesssim \|u\|_{H^{s_1,b_1}} \|v\|_{H^{s_2,b_2}} $$ 
holds, provided the following conditions are satisfied:
\begin{align*}
\nonumber
& b_0 + b_1 + b_2 > \frac{1}{2} \, ,
& b_0 + b_1 \ge 0 \, ,\quad \qquad  
& b_0 + b_2 \ge 0 \, ,
& b_1 + b_2 \ge 0
\end{align*}
\begin{align*}
\nonumber
&s_0+s_1+s_2 > 2 -(b_0+b_1+b_2) \\
\nonumber
&s_0+s_1+s_2 > \frac{3}{2} -\min(b_0+b_1,b_0+b_2,b_1+b_2) \\
\nonumber
&s_0+s_1+s_2 > 1 - \min(b_0,b_1,b_2) \\
\nonumber
&s_0+s_1+s_2 > 1 \\
 &(s_0 + b_0) +2s_1 + 2s_2 > \frac{3}{2} \\
\nonumber
&2s_0+(s_1+b_1)+2s_2 > \frac{3}{2} \\
\nonumber
&2s_0+2s_1+(s_2+b_2) > \frac{3}{2}
\end{align*}
\begin{align*}
\nonumber
&s_1 + s_2 \ge \max(0,-b_0) \, ,\quad
\nonumber
s_0 + s_2 \ge \max(0,-b_1) \, ,\quad
\nonumber
s_0 + s_1 \ge \max(0,-b_2)   \, .
\end{align*}
\end{prop}

\begin{prop}[Null form estimates, \cite{ST} ]  
\label{Prop.1.2}
Let $\sigma_0,\sigma_1,\sigma_2,\beta_0,\beta_1,\beta_2 \in \R$. Assume that
\begin{equation*}
\left\{
\begin{aligned} 
 & 0 \le \beta_0 < \frac12 < \beta_1,\beta_2 < 1,
  \\
 & \sum \sigma_i + \beta_0 > \frac32 - (\beta_0 + \sigma_1 + \sigma_2),
  \\
 & \sum \sigma_i > \frac32 - (\sigma_0 + \beta_1 + \sigma_2),
  \\
&  \sum \sigma_i > \frac32 - (\sigma_0 + \sigma_1 + \beta_2),
  \\
  &\sum \sigma_i + \beta_0 \ge 1,
  \\
 & \min(\sigma_0 + \sigma_1, \sigma_0 + \sigma_2, \beta_0 + \sigma_1 + \sigma_2) \ge 0,
\end{aligned}
\right.
\end{equation*} 
and that the last two inequalities are not both equalities. Let
\begin{align}
\nonumber
&{\mathcal F}(B_{\pm_1,\pm_2} (\psi_{1_{\pm_1}}, \psi_{2_{\pm_2}}))(\tau_0,\xi_0) \\
\label{2}
& := \int_{\tau_1+\tau_2= \tau_0\, \xi_1+\xi_2=\xi_0} |\angle(\pm_1 \xi_1,\pm_2 \xi_2)|  \widehat{\psi_{1_{\pm_1}}}(\tau_1,\xi_1)  \widehat{\psi_{2_{\pm_2}}}(\tau_2,\xi_2) d\tau_1 d\xi_1 \, . 
\end{align}
Then we have the null form estimate
$$
  \norm{B_{(\pm_1 \xi_1,\pm_2 \xi_2)}(u,v)}_{H^{-\sigma_0,-\beta_0}}
  \lesssim
  \norm{u}_{X^{\sigma_1,\beta_1}_{\pm_1}} \norm{v}_{X^{\sigma_2,\beta_2}_{\pm 2}}\, .
$$
\end{prop}

The following multiplication law is well-known:
\begin{prop} {\bf (Sobolev multiplication law)}
\label{SML}
Let $n\ge 2$ , $s_0,s_1,s_2 \in \R$ . Assume
$s_0+s_1+s_2 > \frac{n}{2}$ , $s_0+s_1 \ge 0$ ,  $s_0+s_2 \ge 0$ , $s_1+s_2 \ge 0$. Then the following product estimate holds:
$$ \|uv\|_{H^{-s_0}} \lesssim \|u\|_{H^{s_1}} \|v\|_{H^{s_2}} \, .$$
\end{prop}

Finally, we formulate the fundamental theorem which allows to reduce the local well-posedness for a system of nonlinear wave equations to suitable estimates for the nonlinearities. It is essentially contained in the paper by \cite{KS}. 

\begin{prop}
\label{Prop1.6} 
Let $u_0 \in H^s$ , $u_1 \in H^{s-1}$ , $v_0 \in H^r$ , $v_1 \in H^{r-1}$ , $\psi_{{\pm}_0} \in H^l$ be given. Assume that
\begin{align*}
\|{\mathcal H}_{\pm}(u,\partial u,v , \partial v,\psi_+,\psi_-)\|_{X^{l,b-1+\epsilon}_{\pm}} &\le  \omega_0(\|u\|_{F^s},\|v\|_{G^r},\|\psi_+\|_{X^{l,b}_+},\|\psi_-\|_{X^{l,b}_-}) \, , \\ 
\| \Lambda_+^{-1} \Lambda_-^{\epsilon-1} {\mathcal M}(u,\partial u,v , \partial v,\psi_+,\psi_-)\|_{F^s} & \le \omega_1(\|u\|_{F^s},\|v\|_{G^r},\|\psi_+\|_{X^{l,b}_+},\|\psi_-\|_{X^{l,b}_-}) \, , \\
\| \Lambda_+^{-1} \Lambda_-^{\epsilon-1} {\mathcal N}(u,\partial u,v , \partial v,\psi_+,\psi_-)\|_{G^r} & \le \omega_2(\|u\|_{F^s},\|v\|_{G^r},\|\psi_+\|_{X^{l,b}_+},\|\psi_-\|_{X^{l,b}_-}) \, , 
\end{align*}
and
\begin{align*}
&\|{\mathcal H}_{\pm}(u,\partial u,v , \partial v,\psi_+,\psi_-)- {\mathcal H}_{\pm}(u',\partial u',v' , \partial v',\psi_+',\psi_-')\|_{X^{l,b-1+\epsilon}_{\pm}}\\
&+\| \Lambda_+^{-1} \Lambda_-^{\epsilon-1} ({\mathcal M}(u,\partial u,v , \partial v,\psi_+,\psi_-) - {\mathcal M}(u',\partial u',\partial v',\psi_+',\psi_-'))\|_{F^s} \\
&+
\| \Lambda_+^{-1} \Lambda_-^{\epsilon-1} ({\mathcal N}(u,\partial u,v , \partial v,\psi_+,\psi_-) - {\mathcal N}(u',\partial u',v' , \partial v',\psi_+',\psi_-')\|_{G^r} \\ &\le \omega(\|u\|_{F^s},\|u'\|_{F^s},\|v\|_{G^r}, \|v'\|_{G^r},\|\psi_{\pm}\|_{X_{\pm}^{l,b}})\, \cdot \\
&\quad \cdot(\|u-u'\|_{F^s} + \|v-v'\|_{G^r} +\|\psi_+- \psi_+'\|_{X_+^{l,b}}  + +\|\psi_-- \psi_-'\|_{X_-^{l,b}})             \, , \\
\end{align*}
where $\omega,\omega_0,\omega_1,\omega_2$ are continuous functions with $\omega(0,0,0,0) = \omega_0(0,0,0,0) = \omega_1(0,0,0,0) = \omega_2(0,0,0,0) = 0$.
Then the Cauchy problem
$$ \Box \, u = {\mathcal M}(u,\partial u,v,\partial v,\psi_+,\psi_-) \quad , \quad \Box \, v = {\mathcal N}(u,\partial u,v,\partial v,\psi_+,\psi_-) \,$$
$$ (i\partial_t \pm |\nabla|)\psi_{\pm} = {\mathcal H}_{\pm}(u,\partial u,v , \partial v,\psi_+,\psi_-) $$
with data
$$\psi_{\pm}(0) = \psi_{\pm_0} \, , \, u(0) = u_0 \, , \, (\partial_t u)(0) = u_1 \, , \, v(0)= v_0 \, , \,(\partial_t v)(0) = v_1 $$
is locally well-posed, i.e. , there exists $T>0$ , such that there exists a unique solution $\psi_{\pm} \in X_{\pm}^{l,b}[0,T]$ ,  $u \in F^s_T$ , $v \in G^r_T$ .
\end{prop}
\begin{proof}
This is proved by the contraction mapping principle provided the solution space fulfills suitable assumptions. The case of a single equation $\Box \, u = {\mathcal M}(u,\partial u)$ and the solution space $F^s$ was proven by \cite{KS}, Theorems 5.4 and 5.5, Propositions 5.6 and 5.7. Similarly the well-known case of the equation $                     (i\partial_t \pm |\nabla|)\psi_{\pm} = {\mathcal H}_{\pm}(u,\psi_+,\psi_-)$ is an immediate consequence of \cite{AFS1}, Lemma 5. Our case is a straightforward modification of these results, thus we omit the proof.
\end{proof}

\section{Null structure}
Our aim in what follows is the following proposition.
\begin{prop}
\label{Prop.1}
Let the assumptions on $s,l,r$ of Theorem \ref{Theorem1.1} be satisfied.
Given initial data $a_0^{\nu} \in H^s$ , $a_1^{\nu} \in H^{s-1}$ , $  f_0^{\mu \nu}  \in H^r$ ,  $f_1^{\mu\nu} \in H^{r-1}$ , $ \psi_{\pm}^0 \in H^l$ ,   there exists a time $T > 0$ , which depends on the norms of the data, such that the modified  Cauchy problem (\ref{1.1a}),(\ref{1.1b}),(\ref{1.1c}) with data
\begin{align*}
&A_{\nu}(0) = a_0^{\nu}, \, (\partial_t A_{\nu})(0) =  a_1^{\nu},
        \,
    F_{\mu \nu}(0) =f_0^{\nu}, \, (\partial_t F_{\mu \nu})(0) =  f_1^{\mu \nu}, \\ & \, \psi_{i,\pm}(0) = \psi_{i,\pm}^0 =  \Pi_{\pm} \psi_0^i .
\end{align*}
 has a unique solution $$A \in F^s_T \, , \,F\in G^r_T \, , \, \psi_{\pm} \in X_{\pm}^{l,\frac{3}{4}+}[0,T]$$ (these spaces are defined in Def. \ref{Def.1.2}). The solution depends continously on the data and higher regularity persists.
\end{prop}

{\bf Remark 3:} In section 4 it is shown that the unique solution of Proposition \ref{Prop.1} fulfills the Lorenz gauge condition and also $F=F[A]$ provided the data fulfill the conditions (\ref{f1})-(\ref{f4}) and (\ref{Const}). This proves immediately that it is also the unique solution for the Cauchy problem for the Yang-Mills-Dirac system in Lorenz gauge, so that Theorem \ref{Theorem1.1} follows. \\[0.5em]

The decisive point now is,  that it is possible to replace the nonlinear terms under the assumption that the Lorenz gauge condition holds by nonlinearities containing null forms and prove the necessary estimates for these modified terms. This was done for the Yang-Mills system by \cite{ST} and \cite{Te}. For the nonlinear terms which contain $\psi$ this is done in the sequel as far as possible.

The standard null forms are given by
\begin{equation}\label{OrdNullforms}
\left\{
\begin{aligned}
Q_{0}(u,v)&=\partial_\alpha u \partial^\alpha v=-\partial_t u \partial_t v+\partial_i u \partial^j v,
\\
Q_{\alpha\beta}(u,v)&=\partial_\alpha u \partial_\beta v-\partial_\beta u \partial_\alpha v.
\end{aligned}
        \right.
\end{equation}
For $\mathfrak g$-valued $u,v$, define a commutator version of null forms by 
\begin{equation}\label{CommutatorNullforms}
\left\{
\begin{aligned}
  Q_0[u,v] &= [\partial_\alpha u, \partial^\alpha v] = Q_0(u,v) - Q_0(v,u),
  \\
  Q_{\alpha\beta}[u,v] &= [\partial_\alpha u, \partial_\beta v] - [\partial_\beta u, \partial_\alpha v] = Q_{\alpha\beta}(u,v) + Q_{\alpha\beta}(v,u).
\end{aligned}
\right.
\end{equation}

 Note the identity
\begin{equation}\label{NullformTrick}
  [\partial_\alpha u, \partial_\beta u]
  = \frac12 \left( [\partial_\alpha u, \partial_\beta u] - [\partial_\beta u, \partial_\alpha u] \right)
  = \frac12 Q_{\alpha\beta}[u,u].
\end{equation}

Define 
\begin{equation}\label{NewNull} 
  \mathcal{Q}[u,v] = - \frac12 Q_{jk}\left[\Lambda^{-1}(R^j u^k - R^k u^j),v\right]
  - Q_{0j}\left[R^j u_0, v \right],
\end{equation}
where 
$R_i = \Lambda^{-1}\partial_i $ is the Riesz transform.

We follow Tesfahun \cite{Te} in the following generalizing his 3-dimensional results to arbitrary dimension $n \ge 3$.

We split the spatial part $\mathbf A=(A_1,..., A_n)$ of the potential into divergence-free and curl-free parts and a smoother part:
\begin{equation}\label{SplitA}  
\mathbf A =  A^{df} + A^{cf} + \angles{\nabla}^{-2} \mathbf A,
\end{equation}
where
\begin{align*}
  (A^{df})^j &= R^k(R_j A_k - R_k A_j) \, ,
  \\
  (A^{cf})^j &= -R_j R_k A^k \, .
\end{align*}

\begin{lemma}
\label{Lemma1Te}
(cf. \cite{Te},Lemma 1)
In the Lorenz gauge we have the identities

\begin{align}
\label{2.5}
[A^{\alpha},\partial_{\alpha} \phi] &= \mathcal{Q} [\Lambda^{-1} \mathbf A, \phi] + [\Lambda^{-2} A^{\alpha},\partial_{\alpha} \phi ] \, , \\
\label{2.6}
[\partial_t A^{\alpha}, \partial_{\alpha} \phi] &= Q_{0i}[A^i,\phi] \, .
\end{align}
\end{lemma}
\begin{proof}
Writing
$$ A^{\alpha} \partial_{\alpha} \phi = (-A_0 \partial_t \phi + A^{cf} \cdot \nabla \phi) + A^{df} \cdot \nabla \phi + \Lambda^{-2} \mathbf{A} \cdot \nabla \phi $$
one easily checks using the Lorenz gauge $\partial_t A_0 = \Lambda R_k A^k$ :
\begin{align*}
A^{cf} \cdot \nabla \phi &= - R_j R_k A^k \partial^j \phi = -\partial_t \Lambda^{-1} R_j A_0 \partial^j \phi \\
A_0 \partial_t \phi & = -\Lambda^{-2} \partial_j \partial^j A_0 \partial_t \phi + \Lambda^{-2} A_0 \partial_t \phi \\
& = - \partial_j(\Lambda^{-1} R^j A_0) \partial_t \phi + \Lambda^{-2} A_0 \partial_t \phi \, , 
\end{align*}
so that
$$ -A_0 \partial_t \phi + A^{cf} \cdot \nabla \phi = - Q_{0j}(\Lambda^{-1} R^j A_0,\phi) - \Lambda^{-2} A_0 \partial_t \phi \, . $$
Next
\begin{align*}
A^{df} \cdot \nabla \phi &= R^k(R_j A_k - R_k A_j)\partial^j \phi \\
&= \Lambda^{-2} \partial^k \partial_j A_k \partial^j \phi + A_j \partial^j \phi \\
&= -\half \left( \Lambda^{-2}(\partial_j \partial^j A_k - \partial_j \partial_k A^j)\partial^k \phi - \Lambda^{-2}(\partial^k \partial_j A_k - \partial_k \partial^k A_j) \partial^j \phi \right) \\
&= -\half \left(\partial_j \Lambda^{-1}(R^j  A_k - R_k A^j)\partial^k \phi - \partial^k \Lambda^{-1}(R_j A_k - R_k  A_j) \partial^j \phi \right) \\
&= - \half Q_{jk}(\Lambda^{-1}(R^j A^k - R^k A^j),\phi) \, .
\end{align*}
This leads to (\ref{2.5}). For (\ref{2.6}) we use the Lorenz gauge to obtain
\begin{align*}
[\partial_t A^{\alpha}, \partial_{\alpha} \phi] = [-\partial_t A_0,\partial_t \phi] + [\partial_t
 A^i, \partial_i \phi] 
= - [\partial_i A^i, \partial_i \phi] + [\partial_t A^i, \partial_i \phi]= Q_{0i}[A^i,\phi] \, .
\end{align*}
\end{proof}

\begin{lemma}
\label{Lemma2Te}
(cf. \cite{Te},Lemma 2)
In the Lorenz gauge the following identity holds:
$$ [A^{\alpha} , \partial_{\beta} A_{\alpha}] = \sum_{i=1}^4 \Gamma_{\beta}^i(A,\partial A,F,\partial F) \, , $$
where 
\begin{align*}
\Gamma_{\beta}^1 &= -[A_0,\partial_{\beta} A_0] + [\Lambda^{-1}R_j(\partial_t A_0), \Lambda^{-1} R^j \partial_t (\partial_{\beta} A_0)] \, , \\
\Gamma_{\beta}^2 &= \sum_{i,j} Q_{ij}[\Lambda^{-1}R_k A^k, \Lambda^{-1} R_j \partial_{\beta} A_i] + \sum_{i,j} Q_{ij}[\Lambda^{-1} R_k \partial_{\beta}A^k, \Lambda^{-1} R_j A_i] \, , \\
\Gamma_{\beta}^3 &= \sum_j \Big([\Lambda^{-1} R^i F_{ji},\Lambda^{-1} R^k \partial_{\beta} F_{jk}] + [\Lambda^{-1} R^i F_{ji}, \Lambda^{-1} \partial_{\beta} R^k[A_k,A_j]] \\
& \hspace{1em} + [\Lambda^{-1} R^i [A_i,A_j], \Lambda^{-1} \partial_{\beta} R^k F_{jk}] + [\Lambda^{-1}R^i [A_i,A_j], \Lambda^{-1} \partial_{\beta} R^k [A_k,A_j]] \Big) \, , \\
\Gamma_{\beta}^4 &= [\Lambda^{-2} \mathbf{A}, \partial_{\beta} \mathbf{A}] +[\mathbf{A}, \Lambda^{-2} \partial_{\beta} \mathbf{A}] \, .
\end{align*}
\end{lemma}
\begin{proof}
We write
$$ [A^{\alpha} , \partial_{\beta} A_{\alpha}] = -[A_0,\partial_{\beta} A_0] + [A^j,\partial_{\beta}A_j] = \sum_{i=1}^4 \Gamma_{\beta}^i(A,\partial A,F,\partial F) \, , $$
where
\begin{align*}
\Gamma_{\beta}^1 &= - [A_0,\partial_{\beta} A_0] + [A^{cf},\partial_{\beta} A^{cf}] \, , \\
\Gamma_{\beta}^2 &= [A^{cf},\partial_{\beta} A^{df}] + [A^{df},\partial_{\beta} A^{cf} ] \, , \\
\Gamma_{\beta}^3 &= [A^{df},\partial_{\beta} A^{df}] \, 
\end{align*}
and $\Gamma_{\beta}^4$ as above.

For $\Gamma_{\beta}^1$ we use $\partial_t A_0 = \Lambda R_k A^k$ and obtain
\begin{align*}
-A_0 \partial_{\beta} A_0 + A^{cf} \partial_{\beta} A^{cf} &= -A_0 \partial_{\beta}A_0 + R_j R_k A^k \partial_{\beta} R^j R_k A^k \\
&= -A_0 (\partial_{\beta} A_0) + \Lambda^{-1} R_j (\partial_t A_0)  \Lambda^{-1} \partial_t(\partial_{\beta} A_0) \, ,
\end{align*}  
which gives the result. Concerning $\Gamma_{\beta}^2$ we obtain
\begin{align*}
A^{cf} \partial_{\beta} A^{df} & = -R^j (R_k A^k) \partial_{\beta} R^i(R_j A_i-R_i A_j) \\
& = -R^j(R_k A^k) R^i(\partial_{\beta} R_j A_i) + R^i(R_k A^k) R^j (\partial_{\beta} R_j A_i) \\
&= \sum_{i,j} Q_{ij}(\Lambda^{-1}R_k A^k,\Lambda^{-1}R_j \partial_{\beta} A_i) \, ,
\end{align*}
which gives the claimed result. For $\Gamma_{\beta}^3$ we use 
$$ F_{ji} := \partial_j A_i - \partial_i A_j + [A_j,A_i] \, , $$
so that
$$ (A^{df})_j = R^i(R_j A_i - R_i A_j) = \Lambda^{-1} R^i F_{ji} + R^i(A_i A_j - A_j A_i) \, . $$
This implies
\begin{align*}
&(A^{df})^j \partial_{\beta} A^{df}_j \\
& = \sum_j \Lambda^{-1} (R^i F_{ji} + R^i(A_iA_j-A_j A_i)) \partial_{\beta} \Lambda^{-1} (R^k F_{jk} + R^k(A_k A_j - A_j A_k)) \\
& = \sum_j \Big(\Lambda^{-1}R^i F_{ji} \Lambda^{-1} R^k \partial_{\beta} F_{jk} + \Lambda^{-1} R^i F_{ji}\Lambda^{-1} \partial_{\beta} R^k ( A_k A_j - A_j A_k) \\
& \hspace{3em}+ \Lambda^{-1} R^i (A_i A_j - A_j A_i) \Lambda^{-1} \partial_{\beta} R^k F_{jk} \\
& \hspace{3em} +  \Lambda^{-1} R^i (A_i A_j - A_j A_i) \Lambda^{-1} \partial_{\beta} R^k (A_k A_j - A_j A_k) \Big) \, .
\end{align*}
Thus we obtain the claimed result.
\end{proof}

The null forms above satisfy the following estimates.
\begin{lemma}
\label{Lemma2.2}
The following estimates hold for $0 \le \alpha \le 1 $ and $Q=Q_{0i}$ or $Q=Q_{ij}$:
\begin{align}
\label{41}
Q_0(u,v) &\precsim D_+^{1-\alpha} D_-^{1-\alpha} (D_+^{\alpha} u D_+^{\alpha} v) + (D_+ D_-^{1-\alpha} u)(D_+^{\alpha} v) + (D_+^{\alpha} u)(D_+ D_-^{1-\alpha} v) \\
\nonumber
Q_0(u,v) &\precsim  D_-^{1-\alpha} (D_+ u D_+^{\alpha} v) +   D_-^{1-\alpha} (D_+^{\alpha} u D_+ v)   \\
 \label{41'} &  \hspace{1em}  + (D_+ D_-^{1-\alpha} u)(D_+^{\alpha} v) + (D_+^{\alpha} u)(D_+ D_-^{1-\alpha} v) \\
\label{42}
Q(u,v) &\precsim D_+^\half D_-^\half (D_+^\half u D_+^\half v) + D_+^\half (D_+^\half D_-^\half u D_+^\half v) + D_+^\half (D_+^\half u D_+^\half D_-^\half v) \\
\nonumber
Q(u,v) &\precsim D_+^{\half-2\epsilon} D_-^{\half-2\epsilon} (D_+^{\half+2\epsilon} u D_+^{\half+2\epsilon} v)
+ D_+^{\half -2\epsilon}(D_+^{\half+2\epsilon} D_-^{\half-2\epsilon} u D_+^{\half+2\epsilon} v ) \\
\label{42'}
& \hspace{1em} + D_+^{\half -2\epsilon}(D_+^{\half+2\epsilon} u D_+^{\half+2\epsilon} D_-^{\half-2\epsilon} v) \\
\nonumber
Q(u,v) &\precsim  D_-^{\half-2\epsilon} (D_+ u D_+^{\half+2\epsilon} v) + D_-^{\half-2\epsilon} (D_+^{\half+2\epsilon} u D_+ v)
+ (D_+ D_-^{\half-2\epsilon} u)( D_+^{\half+2\epsilon} v ) \\
\nonumber
& \hspace{1em} + (D_+^{\half+2\epsilon}  u)( D_+ D_-^{\half-2\epsilon}v )
+ (D_+^{\half+2\epsilon}D_-^{\half-2\epsilon}  u)( D_+ v ) \\
& \hspace{1em}  +(D_+  u)( D_+^{\half+2\epsilon} D_-^{\half-2\epsilon} v )
\label{42''}
\end{align}
\end{lemma}
\begin{proof}
(\ref{41}) is Lemma 7.6 in \cite{KS}, and (\ref{42}) follows immediately from \cite{KMBT}, Prop. 1.
(\ref{42'}) follows by interpolating the estimate for the symbol $q=q(\tau,\xi,\lambda,\eta)$ of \cite{KMBT}, Prop. 1 which led to (\ref{42}) with its trivial bound $ q \lesssim (|\tau|+|\xi|)(|\lambda| + |\eta|) $.  (\ref{41'}) and (\ref{42''}) follow by the fractional Leibniz rule for $\Lambda_+$ and $D_+$ from (\ref{41}) and (\ref{42'}), respectively.
\end{proof}

Next we consider the term $\Gamma_{\beta}^1$ . We may ignore its matrix form and treat 
$$\Gamma^1_k(A_0,, \partial_k A_0)
=-A_0 (\partial_k A_0) + 
\Lambda^{-1} R_j (\partial_t A_0)  \Lambda^{-1} R^j  \partial_t (\partial_k A_0))$$
for $k=1,...,n$ and
\begin{align*}
\Gamma^1_0(A_0,, \partial^i A_i)
&=-A_0 (\partial_0 A_0) + 
\Lambda^{-1} R_j (\partial_t A_0)  \Lambda^{-1} R^j  \partial_t (\partial_0 A_0)) \\
& = -A_0 (\partial^i A_i) + 
\Lambda^{-1} R_j (\partial_t A_0)  \Lambda^{-1} R^j  \partial_t (\partial^i A_i)) \,,
\end{align*}
where we used the Lorenz gauge $\partial_0 A_0 = \partial^i A_i$ in the last line in order to eliminate one time derivative. Thus we have to consider
$$ \Gamma^1(u,v) = -uv + \Lambda^{-1} R_j (\partial_t u) \Lambda^{-1} R^j(\partial_t v) \, , $$
where $u=A_0$ and $v=\partial^i A_i$ or $v=\partial_k A_0$ .

The proof of the following theorem was essentially given by Tesfahun \cite{Te}. In fact the detection of this null structure was the main progress of his paper over Selberg-Tesfahun \cite{ST}.  
\begin{lemma}
\label{Lemma2.1}
The following estimates hold:
\begin{align}
\label{45'}
\Gamma^1(u,v) &\precsim \Gamma^1_1(u,v) + \Gamma^1_2(u,v) + (\Lambda^{-2}u)v + u(\Lambda^{-2}v) \, , \\
\label{45''}
\Gamma^1(u,v) &\precsim uv + \Gamma^1_2(u,v)  \, , 
\end{align}
where
\begin{align}
\nonumber
\Gamma^1_1(u,v)&= D^{\half -2\epsilon} D_-^{\half-2\epsilon} ( D^{-\half +2\epsilon} u \, D^{-\half +2\epsilon} v) + D^{\half-2\epsilon}(D^{-\half +2\epsilon} D_-^{\half -2\epsilon} u D^{-\half+2\epsilon} v ) \\
 \label{43}
& \; \; + D^{\half-2\epsilon}(D^{-\half +2\epsilon}  u \,D^{-\half+2\epsilon} D_-^{\half -2\epsilon} v ) \\ \nonumber
\Gamma^1_2(u,v) &= D_+^{\half -2\epsilon} D_-^{\half-2\epsilon} ( D_+^{\half +2\epsilon} \Lambda^{-1} u D_+^{\half +2\epsilon}\Lambda^{-1} v) + D_+ D_-^{\half-2\epsilon} \Lambda^{-1} u\, D_+^{\half+2\epsilon} \Lambda^{-1}v \\
 \label{44}
& \;\; + D_+^{\half+2\epsilon} \Lambda^{-1}  u \,D_+ D_-^{\half-2\epsilon} \Lambda^{-1} v \\ 
\nonumber
&\lesssim  D_-^{\half-2\epsilon} ( D_+ \Lambda^{-1} u D_+^{\half +2\epsilon}\Lambda^{-1} v) +
 D_-^{\half-2\epsilon} ( D_+^{\half +2\epsilon} \Lambda^{-1} u D_+ \Lambda^{-1} v) \\ 
\label{45}
& \;\; + D_+ D_-^{\half-2\epsilon} \Lambda^{-1} u\, D_+^{\half+2\epsilon} \Lambda^{-1}v
 + D_+^{\half+2\epsilon} \Lambda^{-1}  u \,D_+ D_-^{\half-2\epsilon} \Lambda^{-1} v \, .
\end{align}
\end{lemma}
\begin{proof}
$\Gamma^1(u,v)$ has the symbol
\begin{align*}
p(\xi,\tau,\eta,\lambda)& = -1 + \frac{\angles{\xi,\eta} \tau \lambda}{\angles{\xi}^2 \angles{\eta}^2} = \left( -1 + \frac{\angles{\xi,\eta} \angles{\xi,\eta}}{\angles{\xi}^2 \angles{\eta}^2} + \frac{(\tau \lambda -\angles{\xi,\eta}) \angles{\xi,\eta}}{\angles{\xi}^2 \angles{\eta}^2} \right) = I + II
\end{align*}
Now we estimate
\begin{align*}
|I| & = \left| \frac{|\xi|^2 |\eta|^2\cos^2 \angle(\xi,\eta)}{\angles{\xi}^2 \angles{\eta}^2} -1 \right| \\
&
\le \left| \frac{\angles{\xi}^2 \angles{\eta}^2 \cos^2 \angle(\xi,\eta)}{\angles{\xi}^2 \angles{\eta}^2} -1 \right|
+ \left| \frac{|\xi|^2 |\eta|^2 - \angles{\xi}^2 \angles{\eta}^2}{\angles{\xi}^2 \angles{\eta}^2} \right| \\
& = \sin^2 \angle(\xi,\eta) + \left| \frac{|\xi|^2 |\eta|^2 - \angles{\xi}^2 \angles{\eta}^2}{\angles{\xi}^2 \angles{\eta}^2} \right| \, ,
\end{align*}
where $\angle(\xi,\eta)$ denotes the angle between $\xi$ and $\eta$ .
We have
$$\left| \frac{|\xi|^2 |\eta|^2 - \angles{\xi}^2 \angles{\eta}^2}{\angles{\xi}^2 \angles{\eta}^2} \right| = \frac{|\xi|^2 + |\eta|^2 + 1}{\angles{\xi}^2 \angles{\eta}^2} \le \frac{1}{\angles{\xi}^2} + \frac{1}{\angles{\eta}^2} $$
and
\begin{align*}
&\sin^2 \angle(\xi,\eta) \le |\sin \angle(\xi,\eta)|^{1-4\epsilon} = |1-\cos \angle(\xi,\eta)|^{\half -2\epsilon} |1+\cos \angle(\xi,\eta)|^{\half - 2\epsilon} \\
& \lesssim \frac{|\xi + \eta|^{\half -2\epsilon}}{|\xi|^{\half -2\epsilon} |\eta|^{\half - 2\epsilon}} \left( ||\tau|-|\xi||^{\half-2\epsilon} + ||\lambda|-|\eta||^{\half-2\epsilon} + ||\tau + \lambda| - |\xi + \eta|^{\half-2\epsilon} \right)
\end{align*}
for $0 \le \epsilon \le \frac{1}{4}$
by \cite{KMBT}, Proof of proposition 1. Thus the operator belonging to the symbol I is controlled by $\Gamma^1_1(u,v) + (\Lambda^{-2}u)v + u (\Lambda^{-2} v)$ . 
Moreover 
$$ |II| \le \frac{|\tau \lambda - \angles{\xi,\eta}|}{\angles{\xi} \angles{\eta}} \, .$$ 
This is the symbol of $Q_0(\Lambda^{-1}u,\Lambda^{-1}v)$ , which is controlled by $\Gamma^1_2(u,v)$
by (\ref{41}). Thus we obtain (\ref{45'}) and using the trivial bound $|I| \lesssim 1$  also (\ref{45''}). Finally, (\ref{45}) follows by the fractional Leibniz rule for $D_+$ from (\ref{44}).
\end{proof}

\section{Reduction of the problem to multilinear estimates}
For the pure Yang-Mills system the reformulation and the reduction of our main theorem to nonlinear estimates is completely taken over from Tesfahun \cite{Te} (cf. also the fundamental paper by Selberg and Tesfahun \cite{ST}).

The system (\ref{1.1a}),(\ref{1.1b}),(\ref{1.1c}) in Lorenz gauge may be written in the following form by use of Lemma \ref{Lemma1Te} and (\ref{NullformTrick})  for (\ref{1.1c}), and Lemma \ref{Lemma1Te} and Lemma \ref{Lemma2Te} for (\ref{1.1b}):
\begin{equation}\label{AF}
\begin{aligned}
(i \partial_t \pm | \nabla |  ) \psi_{i,\pm} & = H_{i,\pm}(A,\psi) \\
  \square A_\beta &=  K_\beta(A,F) + J_{\beta}(\psi),
  \\
  \square F_{\beta\gamma} &=  L_{\beta\gamma}(A,F) + I_{\beta\gamma}(A,\psi),
\end{aligned}
\end{equation}
where
\begin{align*}
  K_\beta(A,\partial_t A,F,\partial_t F) &= -2 \mathcal Q[\Lambda^{-1} A,A_\beta] +
  \sum_{i=1}^4\Gamma^i_\beta(A, \partial A, F, \partial F)-2[\Lambda^{-2}  A^\alpha, \partial_\alpha A_\beta ]
  \\
 &\quad  - [A^\alpha, [A_\alpha, A_\beta]],
   \end{align*}
 \begin{align*}
  L_{ij}(A,\partial_t A,F,\partial_t F)
  = &- 2\mathcal Q[\Lambda^{-1} A,F_{ij}]
  + 2\mathcal Q[\Lambda^{-1} \partial_j A, A_i]- 2\mathcal Q[\Lambda^{-1} \partial_i A, A_j] 
  \\
  & + 2Q_0[A_i , A_j]
  + Q_{ij}[A^\alpha,A_\alpha]-2[\Lambda^{-2}  A^\alpha, \partial_\alpha F_{ij} ]
  \\
  &+2[\Lambda^{-2}  \partial_jA^\alpha, \partial_\alpha A_{i} ]-2[\Lambda^{-2}  \partial_i A^\alpha, \partial_\alpha A_{j} ]
  \\
  & - [A^\alpha,[A_\alpha,F_{ij}]] + 2[F_{\alpha i},[A^\alpha,A_j]] - 2[F_{\alpha j},[A^\alpha,A_i]]
  \\
  & - 2[[A^\alpha,A_i],[A_\alpha,A_j]],
  \end{align*} 
  \begin{align*}
  L_{0i}(A,\partial_t A,F,\partial_t F)
  = &- 2\mathcal Q[\Lambda^{-1} A,F_{0i}]
  + 2\mathcal Q[\Lambda^{-1} \partial_i A, A_0]- 2 Q_{0j}[A^j,A_i] \\
	& + 2Q_0[A_0 , A_i]
  + Q_{0i}[A^\alpha,A_\alpha]-2[\Lambda^{-2}  A^\alpha, \partial_\alpha F_{0i} ]\\
	& +2[\Lambda^{-2}  \partial_i A^\alpha, \partial_\alpha A_{0} ]
  - [A^\alpha,[A_\alpha,F_{0i}]] + 2[F_{\alpha 0},[A^\alpha,A_i]] \\
	&- 2[F_{\alpha i},[A^\alpha,A_0]]
  - 2[[A^\alpha,A_0],[A_\alpha,A_i]]
  \end{align*}
  where $\Gamma_{\beta}^i$ are defined in Lemma \ref{Lemma2Te}.\\[0.5em]

By Proposition  \ref{Prop1.6} it is posssible to reduce the proof of the local well-posedness result Proposition \ref{Prop.1} to the following multilinear estimates:
\begin{align}
\label{3.1}
\|\Lambda_+^{-1} \Lambda_-^{\epsilon-1}J_{\nu}(\psi)\|_{F^s} & \lesssim \omega_1(\|\psi\|_{X^{l,\frac{3}{4}+}_{\pm}}) \, , \\
\label{3.2}
\| \Lambda_+^{-1} \Lambda_-^{\epsilon-1} K_{\nu}(A,F)\|_{F^s} & \lesssim \omega_2(\|A\|_{F^s},\|F\|_{G^r}) \, , \\
\label{3.3}
\| \Lambda_+^{-1} \Lambda_-^{\epsilon-1} L_{\mu \nu}(A,F)\|_{G^r} & \lesssim \omega_3(\|A\|_{F^s},\|F\|_{G^r}) \, , \\
\label{3.4}
\| \Lambda_+^{-1} \Lambda_-^{\epsilon-1} I_{\nu}(A,\psi)\|_{G^r} & \lesssim \omega_4(\|A\|_{F^s},\|\psi\|_{X^{l,\frac{3}{4}+}_{\pm}}) \, , \\
\label{3.5}
\|H_{i,\pm}(A,\psi)\|_{X^{l,-\frac{1}{4}++}_{\pm}} & \lesssim \omega_5(\|A\|_{F^s},\|\psi\|_{X^{l,\frac{3}{4}+}_{\pm}}) \, ,
\end{align}
where $\omega_j$ are polynomials with $\omega_1(0)=0$ , $\omega_j(0,0) = 0$ ($j=1,2,3$) . \\[0.5em]

We start by considering the pure Yang-Mills part, namely  (\ref{3.2}) and (\ref{3.3}).
Looking at the terms in $K_{\nu}$ and $L_{\mu \nu}$ and noting the fact that the Riesz transforms 
$R_i$ are bounded in the spaces involved, the estimates
reduce to proving (we remark, that due to the multilinear character of the nonlinearity the estimates for the difference can be treated exactly like the other estimates) .\\
1. the corresponding estimates for the null forms $Q_{ij}$ , $Q_0$ and $Q \in \{Q_{0i},Q_{ij}\}$ :
\begin{align}
  \label{24}
  \norm{\Lambda_+^{-1} \Lambda_-^{\epsilon -1} Q[\Lambda^{-1} A, A]}_{F^s}
  &\lesssim \|A\|_{F^s} \|A\|_{F^s},
  \\
    \label{25}
  \norm{ \Lambda_+^{-1} \Lambda_-^{\epsilon -1} Q_{ij}[\Lambda^{-1} A, \Lambda^{-1} \partial A]}_{F^s}
  &\lesssim  \|A\|_{F^s} \|A\|_{F^s} ,
  \\
 \label{26}
  \norm{\Lambda_+^{-1} \Lambda_-^{\epsilon -1} Q[\Lambda^{-1}A, F]}_{G^r }
  &\lesssim \|A\|_{F^s}  \|F\|_{G^r},
    \\
  \label{27}
  \norm{ \Lambda_+^{-1} \Lambda_-^{\epsilon -1} Q[ A,   A]}_{G^r }
  &\lesssim \|A\|_{F^s} \|A\|_{F^s} ,\\
	\label{28}
  \norm{\Lambda_+^{-1} \Lambda_-^{\epsilon -1}  Q_0[ A,   A]}_{G^r }
  &\lesssim \|A\|_{F^s}  \|A\|_{F^s} ,
	\end{align} 
the following estimate for $\Gamma^1$ and other bilinear terms
\begin{align}
  \label{29}
  \norm{\Lambda_+^{-1} \Lambda_-^{\epsilon -1}\Gamma^1( A, \partial A)}_{F^s}
  &\lesssim\|A\|_{F^s} \|A\|_{F^s} ,
  \\
  \label{30} 
   \norm{\Lambda_+^{-1} \Lambda_-^{\epsilon -1}\Pi( A, \Lambda^{-2} \partial A  ) }_{F^s}
     &\lesssim \|A\|_{F^s} \|A\|_{F^s} ,
     \\
     \label{31} 
   \norm{\Lambda_+^{-1} \Lambda_-^{\epsilon -1} \Pi( \Lambda^{-2} A,  \partial A)   }_{F^s}
     &\lesssim \|A\|_{F^s} \|A\|_{F^s} ,
     \\
  \label{32} 
  \norm{\Lambda_+^{-1} \Lambda_-^{\epsilon -1}\Pi(\Lambda^{-1} F, \Lambda^{-1} \partial F  ) }_{F^s}
     &\lesssim \|F\|_{G^r} \|F\|_{G^r},
     \\
       \label{33} 
   \norm{\Lambda_+^{-1} \Lambda_-^{\epsilon -1} \Pi( \Lambda^{-2} A,  \partial F)   }_{G^r}
     &\lesssim \|A\|_{F^s} \|F\|_{G^r},
     \\
       \label{34} 
   \norm{ \Lambda_+^{-1} \Lambda_-^{\epsilon -1}\Pi( \Lambda^{-1} A,  \partial A)   }_{G^r}
     &\lesssim \|A\|_{F^s} \|A\|_{F^s}
     \end{align}
and\\
2. the following trilinear and quadrilinear estimates:
 \begin{align}
   \label{35}
   \norm{\Lambda_+^{-1} \Lambda_-^{\epsilon -1}\Pi(\Lambda^{-1} F,\Lambda^{-1} \partial( AA) )}_{F^s}
  &\lesssim  \|F\|_{G^r} \|A\|_{F^s} \|A\|_{F^s} ,
  \\
   \label{36}
   \norm{\Lambda_+^{-1} \Lambda_-^{\epsilon -1}\Pi(\Lambda^{-1}\partial F, \Lambda^{-1}  ( AA) )}_{F^s}
  &\lesssim  \|F\|_{G^r} \|A\|_{F^s} \|A\|_{F^s} ,
  \\
  \label{37}
   \norm{\Lambda_+^{-1} \Lambda_-^{\epsilon -1}\Pi(\Lambda^{-1}(AA), \Lambda^{-1} \partial ( AA)) }_{F^s}
  &\lesssim  \|A\|_{F^s} \|A\|_{F^s} \|A\|_{F^s} \|A\|_{F^s} ,
  \\
   \label{38} 
  \norm{\Lambda_+^{-1} \Lambda_-^{\epsilon -1}\Pi(A,A,A)}_{F^s}
  &\lesssim\|A\|_{F^s} \|A\|_{F^s} \|A\|_{F^s},
  \\
   \label{39}
  \norm{\Lambda_+^{-1} \Lambda_-^{\epsilon -1}\Pi(A, A, F)}_{G^r}
  &\lesssim \|A\|_{F^s} 
  \|A\|_{F^s}\|F\|_{G^r},
  \\
   \label{40}
  \norm{\Lambda_+^{-1} \Lambda_-^{\epsilon -1}\Pi(A,A, A, A)}_{G^r}
  &\lesssim  \|A\|_{F^s} \|A\|_{F^s} \|A\|_{F^s} \|A\|_{F^s},
     \end{align}
where $\Pi(\cdots)$ denotes a multilinear operator in its arguments.

\section{Proof of the multilinear estimates}
The obvious embeddings $H^{s,b} \hookrightarrow X^{s,b}_{\pm}$ for $b \le 0$ and $X^{s,b}_{\pm} \hookrightarrow H^{s,b}$ for $b \ge 0$ allow to pass from estimates in $X^{s,b}_{\pm}$ to corresponding estimates in $H^{s,b}$ in the following estimates wherever it is suitable.
\begin{proof}[Proof of (\ref{28})]
 We recall (\ref{41'}) for $\alpha = \epsilon$ :
\begin{align*}
Q_0(u,v) &\precsim  D_-^{1-\epsilon} (D_+ u D_+^{\epsilon} v) +   D_-^{1-\epsilon} (D_+^{\epsilon} u D_+ v)   \\
  &  \hspace{1em}  + (D_+ D_-^{1-\epsilon} u)(D_+^{\epsilon} v) + (D_+^{\epsilon} u)(D_+ D_-^{1-\epsilon} v) \, .
\end{align*}
Thus we have to show the following estimates and remark that we only have to consider the first and third term, because the last two terms are equivalent by symmetry. \\
1. For the first term it suffices to show
\begin{align*}
 &\| \Lambda_+^{-1} \Lambda_-^{-1+\epsilon} \Lambda^{r-1} \Lambda_+  D_-^{1-\epsilon}(D_+ u D_+^{\epsilon} v) \|_{H^{0,\half+\epsilon}} \\
&\hspace{1em}\lesssim \| \Lambda^{s-1} \Lambda_+ u \|_{H^{0,\half+\epsilon}} \| \Lambda^{s-1} \Lambda_+ v \|_{H^{0,\half+\epsilon}} \,. 
\end{align*}
This follows from
$$ \|uv\|_{H^{r-1,\half+\epsilon}} \lesssim \|u\|_{H^{s-1,\half+\epsilon}} \|v\|_{H^{s-\epsilon,\half+\epsilon}} \, , $$
which is a consequence of Prop. \ref{Prop.1.2'} under our conditions $s\ge r$ , $2s-r > \frac{3}{2}$ and $s-1+s-\epsilon > \half+\epsilon$ , thus here we need $s > \frac{3}{4}$ .  \\
2. For the second term we show
\begin{align*}
& \| \Lambda_+^{-1} \Lambda_-^{-1+\epsilon} \Lambda^{r-1} \Lambda_+ (D_+ D_-^{1-\epsilon}u D_+^{\epsilon} v) \|_{H^{0,\half+\epsilon}}
\lesssim \| \Lambda^{s-1} \Lambda_+ u \|_{H^{0,\half+\epsilon}} \| \Lambda^{s-1} \Lambda_+ v \|_{H^{0,\half+\epsilon}} \,.
\end{align*}
Thus it suffices to show
$$\|uv\|_{H^{r-1,-\half+2\epsilon}} \lesssim \|u\|_{H^{s-1,-\half+2\epsilon}} \|v\|_{H^{s-\epsilon,\half+\epsilon}} \, $$
which is a consequence of Prop. \ref{Prop.1.2'} as in 1. under the same assumptions. 
\end{proof}

\begin{proof}[Proof of (\ref{27})]
 We use (\ref{42'}).
Thus we have to show the following estimates and remark that we only have to consider the first two terms, because the last two terms are equivalent by symmetry. \\
1. For the first term it suffices to show
\begin{align*}
 &\| \Lambda_+^{-1} \Lambda_-^{-1+\epsilon} \Lambda^{r-1} \Lambda_+ \Lambda_+^{\half-2\epsilon} \Lambda_-^{\half-2\epsilon}(\Lambda_+^{\half+2\epsilon} u  \Lambda_+^{\half+2\epsilon} v) \|_{H^{0,\half+\epsilon}} \\
&\hspace{1em}\lesssim \| \Lambda^{s-1} \Lambda_+ u \|_{H^{0,\frac{3}{4}+\epsilon}} \| \Lambda^{s-1} \Lambda_+ v \|_{H^{0,\frac{3}{4}+\epsilon}} \,. 
\end{align*}
This follows from
$$ \|uv\|_{H^{r-\half-2\epsilon,0}} \lesssim \|u\|_{H^{s-\half-2\epsilon,\frac{3}{4}+\epsilon}} \|v\|_{H^{s-\half-2\epsilon,\frac{3}{4}+\epsilon}} \, , $$
which is a consequence of Prop. \ref{Prop.1.2'} with parameters $s_0=\half-r+2\epsilon$ , $s_1 = s_2 = s-\half-2\epsilon$ , $b_0 =0$ , $b_1=b_2=\frac{3}{4}+\epsilon$ , so that $s_0+s_1+s_2 > 1 $ , if $2s-r > \frac{3}{2}$ , and $s_0+s_1+s_2+s_1+s_2 > \frac{3}{2}$,  if $4s-r > \frac{11}{4}$ , which holds under our assumptions. \\
2. For the second term we show
\begin{align*}
& \| \Lambda_+^{-1} \Lambda_-^{-1+\epsilon} \Lambda^{r-1} \Lambda_+ \Lambda_+^{\half-2\epsilon} (\Lambda_+^{\half+2\epsilon} \Lambda_-^{\half-2\epsilon}u \Lambda_+^{\half+2\epsilon} v) \|_{H^{0,\half+\epsilon}} \\
&\lesssim \| \Lambda^{s-1} \Lambda_+ u \|_{H^{0,\frac{3}{4}+\epsilon}} \| \Lambda^{s-1} \Lambda_+ v \|_{H^{0,\frac{3}{4}+\epsilon}} \,.
\end{align*}
Using $\Lambda_-^{4\epsilon}u \precsim \Lambda_+^{4\epsilon}u$ ($u \preceq v$ denotes $\abs{\widehat u} \le \widehat v$) it suffices to show
$$\|uv\|_{H^{r-\half+2\epsilon,-\half-2\epsilon}} \lesssim \|u\|_{H^{s-\half-2\epsilon,\frac{1}{4}}} \|v\|_{H^{s-\half-2\epsilon,\frac{3}{4}+\epsilon}} \, $$
which is a consequence of Prop. \ref{Prop.1.2'} as in 1., if $2s-r > \frac{3}{2}$ and $3s-2r > \frac{7}{4}$ , which holds under our assumptions.
\end{proof}

\begin{proof}[Proof of (\ref{32})]
As before it is easy to see that we can reduce to
$$ \|uv\|_{H^{s-1,-\frac{1}{4}+2\epsilon}} \lesssim \|u\|_{H^{r+1,\half+\epsilon}} \|v\|_{H^{r,\half+\epsilon}}\,. $$
This is a consequence of Prop. \ref{Prop.1.2'}. One easily checks that it can be applied  under the conditions $s \le r+1$ , $2r-s > -1$ , $4r-s>-\frac{7}{4}$ and $3r-2s>-2$ , all of which are satisfied under our assumptions.
\end{proof}

\begin{proof}[Proof of (\ref{24}) and (\ref{25})]
We have to prove
$$ \|\Lambda_+^{-1} \Lambda_-^{-1+\epsilon} Q(u,v)\|_{H^{s,\frac{3}{4}+\epsilon}} \lesssim \|\Lambda_+ u\|_{H^{s,\frac{3}{4}+\epsilon}} \|\Lambda_+ v\|_{H^{s,\frac{3}{4}+\epsilon}} \, . $$
By (\ref{42''}) we reduce to the following estimates:
\begin{align*}
\|uv\|_{H^{s-\half+2\epsilon,\frac{1}{4}}} & \lesssim \|u\|_{H^{s+\half-2\epsilon,\frac{3}{4}+\epsilon}} \|v\|_{H^{s-\half-2\epsilon,\frac{3}{4}+\epsilon}} \, , \\
\|uv\|_{H^{s-\half+2\epsilon,-\frac{1}{4}+2\epsilon}} & \lesssim \|u\|_{H^{s+\half-2\epsilon,\frac{1}{4}+3\epsilon}} \|v\|_{H^{s-\half-2\epsilon,\frac{3}{4}+\epsilon}} \, , \\
\|uv\|_{H^{s-\half+2\epsilon,-\frac{1}{4}+2\epsilon}} & \lesssim \|u\|_{H^{s+\half-2\epsilon,\frac{3}{4}+\epsilon}} \|v\|_{H^{s-\half-2\epsilon,\frac{1}{4}+3\epsilon}} \, .
\end{align*}
We apply Proposition \ref{Prop.1.2'}. It is easy to check that the first and second estimate require the assumption $s > \frac{3}{4}$ and the third estimate the assumption $s > \half$ .
\end{proof}

\begin{proof}[Proof of (\ref{29})]
It is sufficient to show
\begin{equation}
\label{G1}
\| \Gamma^1(u,v) \|_{H^{s-1,-\frac{1}{4}+2\epsilon}} \lesssim \|u\|_{H^{s,\frac{3}{4}+\epsilon}} \|v\|_{H^{s-1,\frac{3}{4}+\epsilon}}  \, .
\end{equation}
We use Lemma \ref{Lemma2.1}.\\
a. We first consider $\Gamma_2^1(u,v)$ . By (\ref{44}) it suffices to show the following estimates, all of which are consequences of Proposition \ref{Prop.1.2'}. 
\begin{align*}
\|uv\|_{H^{s-\half-2\epsilon,\frac{1}{4}}} & \lesssim \|u\|_{H^{s+\half-2\epsilon,\frac{3}{4}+\epsilon}} \|v\|_{H^{s-\half-2\epsilon,\frac{3}{4}+\epsilon}} \, ,\\
\|uv\|_{H^{s-1,-\frac{1}{4}+2\epsilon}} & \lesssim \|u\|_{H^{s,\frac{1}{4}}} \|v\|_{H^{s-\half-2\epsilon,\frac{3}{4}+\epsilon}} \, , \\
\|uv\|_{H^{s-1,-\frac{1}{4}+2\epsilon}} & \lesssim \|u\|_{H^{s+\half-2\epsilon,\frac{3}{4}+\epsilon}} \|v\|_{H^{s-1,\frac{1}{4}}} \, . 
\end{align*}
b. Assume that $u$ and $v$ have frequencies $ \ge 1$ ,  so that $\Lambda_+^{\alpha} u \sim D_+^{\alpha}u$ . In this case we use (\ref{45'}) and consider $\Gamma^1_1(u,v)$ . By (\ref{43}) we reduce to 
\begin{align*}
\|uv\|_{H^{s-\half-2\epsilon,\frac{1}{4}}} & \lesssim \|u\|_{H^{s+\half-2\epsilon,\frac{3}{4}+\epsilon}} 
\|v\|_{H^{s-\half-2\epsilon,\frac{3}{4}+\epsilon}} \, , \\
 \|uv\|_{H^{s-\half-2\epsilon,-\frac{1}{4}+2\epsilon}}& \lesssim \| u \|_{H^{s+\half-2\epsilon,\frac{1}{4}+\epsilon}} \|v \|_{H^{s-\half-2\epsilon,\frac{3}{4}+\epsilon}} \, , \\
\|uv\|_{H^{s-\half-2\epsilon,-\frac{1}{4}+2\epsilon}} & \lesssim \|u\|_{H^{s+\half-2\epsilon,\frac{3}{4}+\epsilon}} 
\|v\|_{H^{s-\half-2\epsilon,\frac{1}{4}}} \, .
\end{align*}
These estimates follow from Proposition \ref{Prop.1.2'}, which requires $s > \frac{3}{4}$. \\  
c. Consider $(\Lambda^{-2}u) v$ and $u (\Lambda^{-2}v)$ . It suffices to show
\begin{align*}
\| (\Lambda^{-2}  u)v\|_{H^{s-1,-\frac{1}{4}+2\epsilon}} &\lesssim \|u\|_{H^{s,\frac{3}{4} + \epsilon}} \|v\|_{H^{s-1,\frac{3}{4}+\epsilon}} \, , \\
\|   u(\Lambda^{-2}v)\|_{H^{s-1,-\frac{1}{4}+2\epsilon}} &\lesssim \|u\|_{H^{s,\frac{3}{4} + \epsilon}} \|v\|_{H^{s-1,\frac{3}{4}+\epsilon}}\, .
\end{align*}
Both follow easily from Prop. \ref{SML} under our asumption $s > \half$ . \\
d. Let us now consider the case where the frequencies of $u$ or $v$ are $\le 1$. We use (\ref{45''}) instead of (\ref{45'}). Because $\Gamma^1_2(u,v)$ has already been handled, we only have to consider $uv$ .
If $u$ has low frequencies we obtain by Prop. \ref{SML}:
\begin{align*}
\| uv\|_{H^{s-1,-\frac{1}{4}+2\epsilon}} & \lesssim \|u\|_{H^{\frac{3}{2}+,\frac{3}{4}+\epsilon}} \|v\|_{H^{s-1,\frac{3}{4}+\epsilon}} \\
& \lesssim \|u\|_{H^{s,\frac{3}{4}+\epsilon}} \|v\|_{H^{s-1,\frac{3}{4}+\epsilon}} \, .
\end{align*}
Similarly we treat the case where $v$ has low frequencies. 
\end{proof}

\begin{proof}[Proof of (\ref{26})]
We may reduce to
$$ \|\Lambda_+^{-1} \Lambda_-^{\epsilon -1} \Lambda_+ Q(u,v)\|_{H^{r,\half+\epsilon}} \lesssim \|\Lambda u \|_{H^{s,\frac{3}{4}+\epsilon}} \|\Lambda_+ v\|_{H^{r-1,\half+\epsilon}} \, . $$ 
Now we use (\ref{42'})  and estimate the three terms as follows: \\
1. The estimate for the first term is reduced to (using the trivial estimate $\Lambda_-^{2\epsilon} u \precsim  \Lambda_+^{2\epsilon } u $) :
$$ \|uv\|_{H^{r-\half+2\epsilon,0}} \lesssim \|u\|_{H^{s+\half-2\epsilon,\frac{3}{4}+\epsilon}} \| v \|_{H^{r-\half+2\epsilon,\half+\epsilon}} \, , $$
which follows from Prop. \ref{Prop.1.2'}. \\
2. The estimate for the last term reduces to
$$ \|uv\|_{H^{r-\half+2\epsilon,-\half+2\epsilon}} \lesssim \|u\|_{H^{s+\half-2\epsilon,\frac{3}{4}+\epsilon}} \|v\|_{H^{r-\half +2\epsilon,0}} \, , $$
which is also a consequence of Prop. \ref{Prop.1.2'} under our assumption $2s-r > \frac{3}{2}$ . \\
3. The second term reduces to
$$ \|uv\|_{H^{r-\half+2\epsilon,-\half+2\epsilon}} \lesssim \|u\|_{H^{s+\half-2\epsilon,\frac{1}{4}}}    \|v\|_{H^{r-\half+2\epsilon,\half+2 \epsilon}} \, , $$ 
which requires $s > \frac{3}{4}$ .
\end{proof}

\begin{proof}[Proof of (\ref{30}) and (\ref{31})]
 We need the following estimate
$$
\|uv\|_{H^{s-1,-\frac{1}{4}+\epsilon}}  \lesssim \|u\|_{H^{s,\frac{3}{4}+\epsilon}} \|v\|_{H^{s+1,\frac{3}{4}+\epsilon}} \, ,$$
which easily follows from the Sobolev multiplication law (Prop. \ref{SML}), because $(1-s)+s+(s+1) = s+2 > \frac{3}{2}$ , and (\ref{31}) is treated in the same way. 
\end{proof}

\begin{proof}[Proof of (\ref{33}) and (\ref{34})]
(\ref{33}) reduces to the following estimate
$$ \|uv\|_{H^{r-1,-\half+\epsilon}} \lesssim \|u\|_{H^{s+2,\frac{3}{4}+\epsilon}} \|v\|_{H^{r-1,\half+\epsilon}} \, , $$
which follows from Prop. \ref{SML}, because $s+2 > \frac{3}{2}$ . Similarly (\ref{34}) reduces to
$$ \|uv\|_{H^{r-1,-\half+\epsilon}} \lesssim \|u\|_{H^{s+1,\frac{3}{4}+\epsilon}} \|v\|_{H^{s-1,\frac{3}{4}+\epsilon}} \, , $$
which also holds by Sobolev, where we use our assumption $2s-r+1 > \frac{5}{2}$ .
\end{proof}

\begin{proof}[Proof of (\ref{35})]
We reduce to
$$ \|uvw\|_{H^{s-1,-\frac{1}{4}+2\epsilon}} \lesssim \|u\|_{H^{r+1,\half+\epsilon}} \|vw\|_{H^{r,0}}  \lesssim \|u\|_{H^{r+1,\half+\epsilon}} \|v\|_{H^{s,\frac{3}{4}+\epsilon}} \|w\|_{H^{s,\frac{3}{4}+\epsilon}} \, , $$
which follows from Prop. \ref{Prop.1.2'}, where we use $2r-s > -1$ and $r \ge s-1$ for the first step and $2s-r > \frac{3}{2}$ for the second step.
\end{proof}

\begin{proof}[Proof of (\ref{36})]
We may reduce to
\begin{align*}
 \| u \Lambda^{-1} (vw)\|_{H^{s-1,-\frac{1}{4}+2\epsilon}} & \lesssim \|u\|_{H^{r,\half+\epsilon}} \|\Lambda^{-1} (vw)\|_{H^{s+\half,0}} \\
& \lesssim  \|u\|_{H^{r,\half+\epsilon}} \|v\|_{H^{s,\frac{3}{4}+\epsilon}} \|w\|_{H^{s,\frac{3}{4}+\epsilon}} \, ,
\end{align*}
by our assumption $2r-s > -1$ ,
where we use Prop. \ref{Prop.1.2'} twice .
\end{proof}

\begin{proof}[Proof of (\ref{38})]
 It suffices to consider the case $s=\frac{3}{4}$ . We easily obtain the desired estimate by Prop. \ref{Prop.1.2'}:
$$
 \|uvw\|_{H^{-\frac{1}{4},-\frac{1}{4}+2\epsilon}}  \lesssim \|u\|_{H^{\frac{3}{4},\frac{3}{4}+\epsilon}} \|vw\|_{H^{\frac{1}{4},0}} 
\lesssim \|u\|_{H^{\frac{3}{4},\frac{3}{4}+\epsilon}}  \|v\|_{H^{\frac{3}{4},\frac{3}{4}+\epsilon}}  \|w\|_{H^{\frac{3}{4},\frac{3}{4}+\epsilon}} \, . 
$$
\end{proof}

\begin{proof}[Proof of (\ref{39})]
We obtain by Prop. \ref{Prop.1.2'}:
$$ \|uvw\|_{H^{r-1,-\half+2\epsilon}} \lesssim \|uv\|_{H^{0+,0}}\|w\|_{H^{r,\half+\epsilon}} \lesssim \|u\|_{H^{s,\frac{3}{4}+\epsilon}} \|v\|_{H^{s,\frac{3}{4}+\epsilon}} \|w\|_{H^{r,\half+\epsilon}} \, . $$
\end{proof}

\begin{proof}[Proof of (\ref{37})]
We have to show
$$ \| \Lambda^{-1} (uv) wz \|_{H^{s-1,-\frac{1}{4}+2\epsilon}} \lesssim \|u\|_{H^{s,\frac{3}{4}+\epsilon}} 
\|v\|_{H^{s,\frac{3}{4}+\epsilon}} \|w\|_{H^{s,\frac{3}{4}+\epsilon}} \|z\|_{H^{s,\frac{3}{4}+\epsilon}} \, . $$
It suffices to consider the minimal value $s=\frac{3}{4}$ , which by Proposition \ref{Prop.1.2'}  can be estimated as follows:
\begin{align*}
&\| \Lambda^{-1} (uv) wz \|_{H^{-\frac{1}{4},-\frac{1}{4}+2\epsilon}}  \lesssim \| \Lambda^{-1} (uv)  \|_{H^{\frac{3}{4}+,\half+\epsilon}} \| wz  \|_{H^{0,0}} \\
& \hspace{1em}\lesssim \| u \|_{H^{\frac{3}{4},\frac{3}{4}+\epsilon}} \|v  \|_{H^{\frac{3}{4},\frac{3}{4}+\epsilon}} \| w\|_{H^{\frac{3}{4},\frac{3}{4}+\epsilon}}  \|z  \|_{H^{\frac{3}{4},\frac{3}{4}+\epsilon}} \, . 
\end{align*}
\end{proof}

\begin{proof}[Proof of (\ref{40})]
We have to show
$$ \|uv wz \|_{H^{r-1,-\half+2\epsilon}} \lesssim \|u\|_{H^{s,\frac{3}{4}+\epsilon}} 
\|v\|_{H^{s,\frac{3}{4}+\epsilon}} \|w\|_{H^{s,\frac{3}{4}+\epsilon}} \|z\|_{H^{s,\frac{3}{4}+\epsilon}} \, . $$
By our assumption $2s-r > \frac{3}{2}$ the left hand side is bounded by the term \\
$\|uv wz \|_{H^{2s-\frac{5}{2}-\epsilon,-\half+2\epsilon}}$ . It suffices to prove the remaining estimate for the (minimal) value $s= \frac{3}{4}$ . By Proposition \ref{Prop.1.2'} we obtain
\begin{align*}
\|uv wz \|_{H^{-1,-\half+2\epsilon}} & \lesssim \|uv\|_{H^{-\frac{1}{4},\half+\epsilon}} \| wz \|_{H^{\frac{3}{8},0}} \\
& \lesssim \|u\|_{H^{\frac{3}{4},\frac{3}{4}+\epsilon}}    \|v\|_{H^{\frac{3}{4},\frac{3}{4}+\epsilon}} \| w\|_{H^{\frac{3}{4},\frac{3}{4}+\epsilon}}  \|z \|_{H^{\frac{3}{4},\frac{3}{4}+\epsilon}} \, .
\end{align*}
\end{proof}

\begin{proof}[ Proof of (\ref{3.4}).]
We start with the quadratic terms in $I_{\mu \nu}(A,\psi)$ . We have to prove
\begin{equation}
\nonumber
\| \partial_{\beta} \langle \psi_1, \alpha_{\gamma} \psi_2 \rangle - \partial_{\gamma} \langle \psi_1, \alpha_{\beta} \psi_2 \rangle \|_{H^{r-1,-\half++}} \lesssim \|\psi_1\|_{X^{l,\frac{3}{4}+}_{\pm_1}} \|\psi_2\|_{X^{l,\frac{3}{4}+}_{\pm_2}} \, .
\end{equation}
Using the obvious inequality
$||\tau|-|\xi||^{-\half++} \le |\tau-|\xi||^{-\half++} +|\tau+|\xi||^{-\half++}$ it is sufficient to prove
\begin{equation}
\label{1}
\| \partial_{\beta} \langle \psi_1, \alpha_{\gamma} \psi_2 \rangle - \partial_{\gamma} \langle \psi_1, \alpha_{\beta} \psi_2 \rangle \|_{X_{\pm_0}^{r-1,-\half++}} \lesssim \|\psi_1\|_{X^{l,\frac{3}{4}+}_{\pm_1}} \|\psi_2\|_{X^{l,\frac{3}{4}+}_{\pm_2}} 
\end{equation}
for independent signs $\pm_0$ , $\pm_1$ , $\pm_2$ .

We obtain by (\ref{2.7}) :
\begin{align*}
& \partial_{\beta} \langle \psi_1, \alpha_{\gamma} \psi_2 \rangle - \partial_{\gamma} \langle \psi_1, \alpha_{\beta} \psi_2 \rangle \\
&=  \sum_{\pm_1,\pm_2} (\partial_{\beta} \langle \psi_{1_{\pm_1}}, \alpha_{\gamma} \Pi_{\pm_2} \psi_{2_{\pm_2}} \rangle - \partial_{\gamma} \langle \psi_{1_{\pm_1}}, \alpha_{\beta} \Pi_{\pm_2} \psi_{2_{\pm_2}} \rangle) \\
& = \sum_{\pm_1,\pm_2} (\partial_{\beta} \langle \psi_{1_{\pm_1}}, \Pi_{\mp_2}(\alpha_{\gamma}  \psi_{2_{\pm_2}}) \rangle - \partial_{\gamma} \langle \psi_{1_{\pm_1}}, \Pi_{\mp_2}(\alpha_{\beta}  \psi_{2_{\pm_2}}) \rangle)\\
& \hspace{1em} -  \sum_{\pm_1,\pm_2} (\partial_{\beta} \langle \psi_{1_{\pm_1}}, R^{\gamma}_{\pm_2}  \psi_{2_{\pm_2}} \rangle - \partial_{\gamma} \langle \psi_{1_{\pm_1}}, R^{\beta}_{\pm_2}  \psi_{2_{\pm_2}} \rangle)\\
&= I + II \, .
\end{align*}
Both terms are null forms. We namely obtain
$$ I = |\nabla|  \sum_{\pm_0,\pm_1,\pm_2} (R_{{\beta},\pm_0} \langle \psi_{1_{\pm_1}}, \Pi_{\mp_2}(\alpha_{\gamma}  \psi_{2_{\pm_2}}) \rangle - R_{{\gamma,\pm_0}} \langle \psi_{1_{\pm_1}}, \Pi_{\mp_2}(\alpha_{\beta}  \psi_{2_{\pm_2}}) \rangle) \, . $$
We consider each term separately and remark that $R_{\beta,\pm_0}$ is irrelevant. We obtain
\begin{align}
\label{30a}
&{\mathcal F}(\langle \Pi_{\pm_1} \psi_{1_{\pm_1}},\Pi_{\mp_2}\alpha_{\gamma} \psi_{2_{\pm_2}}\rangle) (\tau_0,\xi_0) \\
\nonumber
 &= \int_{\tau_1+\tau_2=\tau_0 \, , \, \xi_1 + \xi_2= \xi_0} \langle \Pi(\pm_1 \xi_1) \widehat{\psi_{1_{\pm_1}}} (\tau_1,\xi_1),\Pi(\mp_2 \xi_2)\alpha_{\gamma} \Pi(\pm_2 \xi_2) \widehat{\psi_{2_{\pm_2}}}(\tau_2,\xi_2)\rangle d\tau_1 d \xi_1 \\ \nonumber
&= \int_{\tau_1+\tau_2=\tau_0 \, , \, \xi_1 + \xi_2= \xi_0} \langle  \widehat{\psi_{1_{\pm_1}}} (\tau_1,\xi_1),\Pi(\pm_1 \xi_1)\Pi(\mp_2 \xi_2)\alpha_{\gamma}  \Pi(\pm_2 \xi_2) \widehat{\psi_{2_{\pm_2}}}(\tau_2,\xi_2)\rangle d\tau_1 d \xi_1  .
\end{align}
Now we use the estimate
$$ | \Pi(\pm \xi_1) \Pi(\mp \xi_2) z| \lesssim |z| \angle (\pm \xi_1,\pm \xi_2) $$ 
proven by \cite{AFS}, Lemma 2. This implies
$$ I \lesssim B_{\pm_1,\pm_2} (\psi_{1_{\pm_1}}, \psi_{2_{\pm_2}}) \, , $$
where $ B_{\pm_1,\pm_2}$ is defined by (\ref{2}).

We have to prove
\begin{equation}
\label{3}
\| B_{\pm_1,\pm_2} (\psi_{1_{\pm_1}}, \psi_{2_{\pm_2}}) \|_{X_{\pm_0}^{r,-\half++}} \lesssim \|\psi_{1_{\pm_1}} \|_{X^{l,\frac{3}{4}+}_{\pm_1}} \|\psi_{2_{\pm_2}} \|_{X^{l,\frac{3}{4}+}_{\pm_2}} \, .
\end{equation}
Next we obtain
$$ II = |\nabla| \sum_{\pm_1,\pm_2} (R_{\beta,\pm_0} \langle \psi_{1_{\pm_1}}, R^{\gamma}_{\pm_2}  \psi_{2_{\pm_2}} \rangle - R_{\gamma,\pm_0} \langle \psi_{1_{\pm_1}}, R^{\beta}_{\pm_2}  \psi_{2_{\pm_2}} \rangle) \, . $$
By duality we have to prove
\begin{align*}
&\left|\int\left(\langle \psi_{1_{\pm_1}},R^{\gamma}_{\pm_2} \psi_{2_{\pm_2}} \rangle R^{\beta}_{\pm_0} \psi_{0_{\pm_0}} - \langle \psi_{1_{\pm_1}},R^{\beta}_{\pm_2} \psi_{2_{\pm_2}} \rangle R^{\gamma}_{\pm_0} \psi_{0_{\pm_0}} \right) dx dt \right| \\
& \hspace{1em}
\lesssim \|\psi_{1_{\pm_1}}\|_{X^{l,\frac{3}{4}+}_{\pm_1}} \|\psi_{2_{\pm_2}}\|_{X^{l,\frac{3}{4}+}_{\pm_2}} \|\psi_{0_{\pm_0}}\|_{X^{-r,\half--}_{\pm_0}} \, .
\end{align*}
We remark that the left hand side possesses a $Q^{\gamma \beta}$-type null form between $\psi_{2_{\pm_2}}$ and $\psi_{0_{\pm_0}}$ .

It is well-known (cf. e.g. \cite{ST}) that the bilinear form $Q^{\gamma \beta}_{\pm_{1}, \pm_{2}}$ ,  defined by
\begin{align*}
	& Q^{\gamma \beta}_{\pm_{2}, \pm_{0}} (\phi_{2_{\pm_{2}}}, \phi_{0_{ \pm_{0}}}) 
	:= R^{\gamma}_{\pm_{2}} \phi_{2_{\pm_{2}}} R^{\beta}_{\pm_{0}} \phi_{0_{ \pm_{0}}} - R^{\beta}_{\pm_{2}} \phi_{2_{ \pm_{2}}} R^{\gamma}_{\pm_{0}} \phi_{0_{ \pm_{0}}} \, ,
\end{align*}
similarly to the standard null form $Q_{\gamma \beta}$ , which is defined by replacing the modified Riesz transforms $R^{\mu}_{\pm}$ by $\partial^{\mu}$, fulfills the following estimate:
$$  Q^{\gamma \beta}_{\pm_{2}, \pm_{0}} (\phi_{2_{ \pm_{2}}}, \phi_{0_{ \pm_{0}}}) \precsim B_{\pm_2,\pm_0}(\psi_{2_{\pm_2}},\psi_{0_{\pm_0}} )\, . $$
Recall that  $u \precsim v$ is defined by $|\widehat{u}| \lesssim |\widehat{v}|$ .
We have to prove
\begin{equation}
\label{4}
\|B_{\pm_2,\pm_0}(\psi_{2_{\pm_2}},\psi_{0_{\pm_0}}) \|_{X^{-l,-\frac{3}{4}-}_{\pm_1}} \lesssim  \|\psi_{2_{\pm_2}}\|_{X^{l,\frac{3}{4}+}_{\pm_2}} \|\psi_{0_{\pm_0}}\|_{X^{-r,\half--}_{\pm_0}} \, .
\end{equation}
For the estimates (\ref{3}) and (\ref{4}) we apply Prop. \ref{Prop.1.2}.

 For (\ref{3}) we have the parameters $\sigma_0 = -r$ , $\beta_0 = \half--$ , $\sigma_1=\sigma_2 = l$ , $b_1=b_2=\frac{3}{4}+$, so that we require the conditions $l \ge r$ , $l+r \ge 0$ . Moreover we need $(\sigma_0 + \sigma_1 + \sigma_2)+\beta_0 = 2l-r+\half > 1 \,\Leftrightarrow \, 2l-r > \half$ and $2l-r > \frac{3}{2}-(-r+\frac{3}{4}+l) \, \Leftrightarrow \, 3l-2r > \frac{3}{4}$ .

For (\ref{4}) we have $\sigma_0 = l$ , $\beta_0 = \frac{3}{4}+$ , $\sigma_1=l$ , $\sigma_2=-r$ , $\beta_1 = \frac{3}{4}+$, $\beta_2 = \half--$, so that we require $(\sigma_0+\sigma_1+\sigma_2)+\beta_0 = 2l-r+\frac{3}{4} > 1 \, \Leftrightarrow \, 2l-r > \frac{1}{4}$ , $2l-r > \frac{3}{2}-(l+\frac{3}{4}-r) \, \Leftrightarrow \, 3l-2r > \frac{3}{4}$ . These and the remaining conditions can easily been shown to be fulfilled under our asumptions.

Next we consider the cubic terms in $I_{\beta \gamma}(A,\psi)$ . The special structure plays no role so that we only have to prove
\begin{equation}
\label{5}
\| A_{\gamma} \psi_1 \psi_2 \|_{H^{r-1,-\half++}} \lesssim \|A_{\gamma} \|_{H^{s,\frac{3}{4}+}} \|\psi_1\|_{X^{l,\frac{3}{4}+}_{\pm_1}}  \|\psi_2\|_{X^{l,\frac{3}{4}+}_{\pm_2}}\, .
\end{equation}
We apply Prop. \ref{Prop.1.2'} twice and obtain
\begin{align*}
\| A_{\gamma} \psi_1 \psi_2 \|_{H^{r-1,-\half++}} \lesssim \|A_{\gamma} \|_{H^{s,\frac{3}{4}+}} \|\psi_1 \psi_2\|_{H^{m,0}} \lesssim  \|A_{\gamma} \|_{H^{s,\frac{3}{4}+}} \|\psi_1\|_{X^{l,\frac{3}{4}+}_{\pm_1}}  \|\psi_2\|_{X^{l,\frac{3}{4}+}_{\pm_2}}\, .
\end{align*}
The first estimate follows with parameters $s_0=1-r$ , $s_1=s$ , $s_2=m$ , $b_0 = \half--$, $b_1= \frac{3}{4}+$ , $b_2=0$ . We require $m \ge r-1$ , $m+s \ge 0$ , $s_0+s_1+s_2=1-r+s+m > 1 \, \Leftrightarrow \, m >r-s$ and $s_0+s_1+s_2 > \frac{3}{2}-(1-r+s) \, \Leftrightarrow \, m > -\half +2r -2s $ , which is the weaker condition on $m$ . The second estimate with parameters $s_0=m$ , $s_1=s_2=l$, $b_0=0$ , $b_1=b_2=\frac{3}{4}+$  requires $ l \ge m$ , $s_0+s_1+s_2=2l-m > 1 \, \Leftrightarrow \, 2l > m+1$ and $2l-m > \frac{3}{2}-2l$  , which is the weaker condition. Combining these conditions we require $ 2l > r-s+1$ , which is fulfilled, because $2l-r > \half > -s+1$ by our assumptions.
Moreover we require $2l > r$ , $2l> 1-s$ , $l > r-1$ and $l > r-s$ , which are also satisfied under our assumptions.
\end{proof}

\begin{proof}[ Proof of (\ref{3.1}).]
We have to prove the following estimate:
\begin{equation}
\| \langle \psi^i,\alpha_{\nu} T^a_{ij} \psi^j \rangle T_a \|_{H^{s-1,-\frac{1}{4}++}} \lesssim \sum_{\pm_1,\pm_2,i,j} (\|\psi_i\|_{X^{l,\frac{3}{4}+}_{\pm_1}} \|\psi_j\|_{X^{l,\frac{3}{4}+}_{\pm_2}}) \, .
\end{equation}
Up to constants the  term $T^a_{ij}$ is irrelevant. Thus we have to prove
$$
\| \psi_{1_{\pm_1}} \psi_{2_{\pm_2}} \|_{H^{s-1,-\frac{1}{4}++}} \lesssim \|\psi_{1_{\pm_1}} \|_{X^{l,\frac{3}{4}+}_{\pm_1}} \|\psi_{2_{\pm_2}} \|_{X^{l,\frac{3}{4}+}_{\pm_2}} \, .
$$
We use Prop. \ref{Prop.1.2'} with parameters $s_0=1-s$ , $s_1=s_2=l$ , $b_0=\frac{1}{4}--$ , $b_1=b_2= \frac{3}{4}+$ . This requires $l \ge s-1$ and $s_0+s_1+s_2 = 1-s+2l > 1 $ , thus $2l-s > 0$ . We remark that in the case $ s > \frac{3}{4} $ this implies $l > \frac{3}{8}$ . Moreover we need the conditions $2l+1-s > \frac{3}{2}-(\frac{1}{4}+2l)$ $\Leftrightarrow$ $4l-s> 0$ , which follows from $2l-s> 0$ and $l> \frac{3}{8}$ . 
\end{proof}

\begin{lemma}
\label{Lemma}
Define $A^{\pm}_{\mu} = \half(A_{\mu} \mp i |\nabla|^{-1} \partial_t A_{\mu})$ , so that $A_{\mu} = A^+_{\mu} + A^-_{\mu}$ . Let $s\in \R$ , $0 \le b \le 1$ . The following estimate applies:
$$ \|A^{\pm}_{\mu}\|_{X^{s,b}_{\pm}} \lesssim \| \Lambda_+ A_{\mu}\|_{H^{s-1,b}} \, . $$
\end{lemma}
\begin{proof}
We start with the estimate
\begin{align*}
\|A^{\pm}_{\mu}\|_{X^{s,b}_{\pm}} & = \| \langle \xi \rangle^s \langle -\tau \pm |\xi|\rangle^b \tilde{A}^{\pm}_{\mu}(\tau,\xi)\|_{L^2_{\tau \xi}} =  \| \langle \xi \rangle^s\langle -\tau \pm |\xi|\rangle^b (1\pm \frac{\tau}{\langle \xi \rangle})\tilde{A}_{\mu}(\tau,\xi)\|_{L^2_{\tau \xi}} \\
& \lesssim  \| \langle \xi \rangle^{s-1} \langle -\tau \pm |\xi|\rangle^b \langle |\xi|\pm \tau \rangle\tilde{A}_{\mu}(\tau,\xi)\|_{L^2_{\tau \xi}}  \, .
\end{align*}
If $\tau \le 0$ we obtain
$$\langle \tau \pm |\xi|\rangle^b \langle |\xi|\mp \tau \rangle = \langle |\tau| \pm |\xi|\rangle^b \langle |\xi|\mp |\tau| \rangle^{1-b}  \langle |\xi|\mp |\tau| \rangle^b \lesssim  \langle |\tau| + |\xi|\rangle \langle |\xi| - |\tau| \rangle^b \, , $$
and in the case $\tau \ge 0$ :
$$\langle \tau \pm |\xi|\rangle^b \langle |\xi|\mp \tau \rangle = \langle |\tau| \mp |\xi|\rangle^b \langle |\xi|\pm |\tau| \rangle^{1-b}  \langle |\xi|\pm |\tau| \rangle^b \lesssim  \langle |\tau| + |\xi|\rangle \langle |\xi| - |\tau| \rangle^b \, , $$
which imply the claimed result.
\end{proof}

\begin{proof}[Proof of (\ref{3.5}).]
Here we assume the Lorenz gauge condition. Using Lemma \ref{Lemma} 
we have to prove for fixed $a$ and $j$:
\begin{equation}
\label{6}
\| \Pi_{\pm_0} (A^a_{\mu_{\pm}} \alpha^{\mu} T^a_{ij} \psi_j) \|_{X^{l,-\frac{1}{4}++}_{\pm_0}} \lesssim \|\psi_j\|_{X^{l,\frac{3}{4}+}_{\pm_1}} \| A^a_{\mu_{\pm}}\|_{X^{s,\frac{3}{4}+}_{\pm}} \, .
\end{equation}
We omit the irrelevant factor $T^a_{ij}$ and apply (\ref{2.7}), so that
\begin{align*}
\Pi_{\pm_0}(A^a_{\mu} \alpha^{\mu} \psi) & =    \sum_{\pm} (-A_0^a \psi_{\pm} + \Pi_{\pm_0} (A^a_j \alpha^j \Pi_{\pm} \psi)\\
& =   \sum_{\pm} (-A_0^a \psi_{\pm} + \Pi_{\pm_0} (A^a_j \Pi_{\mp} \alpha^j \Pi_{\pm} \psi) - \sum_{\pm} A^a_j R^j_{\pm} \psi_{\pm} \\
& = - \sum_{\pm} (A^a_0 \psi_{\pm} + A^a_j R^j_{\pm} \psi_{\pm}) + \sum_{\pm_1} \Pi_{\pm_0}(A^a_{\mu} \Pi_{\mp_1}(\alpha^{\mu} \psi_{\pm_1})) = II + I \, ,
\end{align*}
where we used  $A^a_0 \Pi_{\mp_1} \alpha^0 \psi_{\pm_1} =0$ . Both terms turn out to be null forms.

Concerning I we have to prove by duality
\begin{align*}
\left| \int \int \langle \Pi_{\pm_0} (A^a_{\mu} \Pi_{\mp_1}(\alpha^{\mu} \psi_{\pm_1})), \psi_0 \rangle dx dt \right| \lesssim \|A^a_{\mu}\|_{X^{s,\frac{3}{4}+}_{\pm_2}} \|\psi_{1_{\pm_1}}\|_{X^{l,\frac{3}{4}}_{\pm_1}} \|\Pi_{\pm_0} \psi_0\|_{X^{-l,\frac{1}{4}--}_{\pm_0}} \, .
\end{align*}
The left hand side equals
$$ \left|\int \int (A^a_{\mu} \langle \Pi_{\mp_1}(\alpha^{\mu} \psi_{\pm_1})), \Pi_{\pm_0} \psi_0 \rangle dx dt \right| \, . $$
Similarly as in the proof of (\ref{3.4}) this term contains a null form between  $\psi_{\pm_1}$ and $\psi_0$ , so that we have to prove
\begin{equation}
\label{7}
\| B_{\pm_1,\pm_0} (\psi_{\pm_1} ,\psi_{0_{\pm_0}}) \|_{X^{-s,-\frac{3}{4}+}_{\pm_2}} \lesssim \|\psi_{\pm_1} \|_{X^{l,\frac{3}{4}+}_{\pm_1}} \|\psi_{0_{\pm_0}} \|_{X^{-l,\frac{1}{4}--}_{\pm_0}} \, ,
\end{equation}
where $ B_{\pm_1,\pm_0}$ is defined by (\ref{2}).

In order to show that II is also a null form we use the Helmholtz decomposition for ${\bf A} = (A_1,A_2,A_3)$ , where we drop the index $a$. Following \cite{ST} (cf. also \cite{Te} and \cite{HO}) we obtain
$${\bf A} = A^{df} + A^{cf} + \langle \nabla \rangle^{-2} {\bf A} \,, $$
where
$$ A^{cf} = - \langle \nabla \rangle^{-2} \nabla(\nabla \cdot {\bf A}) \quad , \quad  A^{df} =\langle \nabla \rangle^{-2} \nabla \times (\nabla \times {\bf A}) \, , $$
so that we obtain
\begin{align*}
-II & = \sum_{\pm_1} (A_0 \psi_{\pm_1} + A_l R^l_{\pm_1} \psi_{\pm_1}) \\
& = \sum _{\pm_1} (A_0 \psi_{\pm_1} + A_l^{cf} R^l_{\pm_1} \psi_{\pm_1}) + \sum_{\pm_1} A^{df}_l R^l_{\pm_1} \psi_{\pm_1} + \sum_{\pm_1} \langle \nabla \rangle^{-2} A_l R^l_{\pm_1} \psi_{\pm_1} \\
& =: I_1 + I_2 +I_3\, .
\end{align*}
$I_3$ is a harmless term, because we only have to prove
$$ \|A_l^{\pm_2} R^l_{\pm_1} \psi_{\pm_1}\|_{X^{l,-\frac{1}{4}++}_{\pm}} \lesssim \|A_l^{\pm_2} \|_{X^{s+2,\frac{3}{4}+}_{\pm_2}} \|\psi_{\pm_1}\|_{X^{l,\frac{3}{4}+}_{\pm_1}} \, , $$
which is certainly true by Sobolev for $s > -\half$ .
Now $I_2 = I_4 +  I_5$ , where 
$ I_4 =\sum_{\pm_1} \tilde{A}^{df}_l R^l_{\pm_1} \psi_{\pm_1}$  and  $I_5 = \sum_{\pm_1} (A_l^{df}- \tilde{A}_l^{df}) R^l_{\pm_1} \psi_{\pm_1} $ ,
and $\tilde{A}^{df}: = |\nabla|^{-2} \nabla \times(\nabla \times {\bf A})$. 
$I_5$ is also harmless, because it behaves like $\langle \nabla \rangle^{-2} A_l \psi_{\pm_1}$ , which can be handled as $I_3$ .
We obtain by \cite{ST1} (originally detected by \cite{KM1}):
$$ I_4 = - \sum_{\pm_1} \epsilon^{lkm} \partial_k w_m R^l_{\pm_1} \psi_{\pm_1} =\sum_{\pm_1} (\nabla w_m \times \frac{\nabla}{|\nabla|}\psi_{\pm_1})^m \, , $$
where $\epsilon^{lkm}$ denotes the Levi-Civita symbol with $\epsilon^{123} = 1$ and $w= |\nabla|^{-2} \nabla \times {\bf A}$.
This shows that $I_4$ is a $Q_{ij}$-type null form between $\nabla w_m$ and $\frac{\nabla}{|\nabla|} \psi_{\pm_1}$ . By the definiton of $w$ we know that $\partial_j w_m = |\nabla|^{-2} \partial_j \partial_k A^l \epsilon^{lkm}$ , so that we only have to prove
\begin{equation}
\label{8}
\| B_{\pm_1,\pm_2}(A^{\pm_2}_j, \psi_{\pm_1})\|_{X^{l,-\frac{1}{4}++}_{\pm}} \lesssim \|A^{\pm_2}_j \|_{X^{s,\frac{3}{4}+}_{\pm_1}} \|\psi_{\pm_1}\|_{X^{l,\frac{3}{4}+}_{\pm_1}} \, . 
\end{equation}
$I_1$ is also a nullform as detected by \cite{ST1}. In order to realize that we use the Lorenz condition to obtain
$$ A^{cf} = -\langle \nabla \rangle^{-2} \nabla(\nabla \cdot {\bf A}) = - \langle \nabla \rangle^{-2} \nabla(\partial_t A_0) \, . $$
 Using $\partial_t A_0 = i \langle \nabla \rangle (A_{0_+} -A_{0_{-}})$ , where $A_0=A_{0_+} +  A_{0_-} $, and $\partial_l = \mp i |\nabla |  R^l_{\pm}$ we obtain
\begin{align*}
I_1 & = \sum_{\pm_1} A_0 \psi_{\pm_1} - \langle \nabla \rangle^{-2} \partial_l \partial_t A_0 R^l_{\pm_1} \psi_{\pm_1} \\
& = \sum_{\pm_1} (A_{0_+} + A_{0-}) \psi_{\pm_1} -i \langle \nabla \rangle^{-1} \partial_l (A_{0_+}-A_{0_-}) R^l_{\pm_1} \psi_{\pm_1} \\
& = \sum_{\pm_1} (A_{0_+} + A_{0-}) \psi_{\pm_1} - (\tilde{R}^l_+ A_{0_+}+ \tilde{R}^{l}_- A_{0_-}) R^l_{\pm_1} \psi_{\pm_1} \\
& = \sum_{\pm,\pm_1} (A_{0_{\pm}} \psi_{\pm_1} - \tilde{R}^l_{\pm} A_{0_{\pm}} R^l_{\pm_1} \psi_{\pm_1} ) \, ,
\end{align*}
where $\tilde{R}^l_{\pm} := \pm i \langle \nabla \rangle^{-1} \partial_l$ .
The Fourier symbol of this expression can be estimated as follows:
\begin{align*}
 \left| 1 - \langle \frac{\pm \xi}{\langle \xi \rangle} , \frac{\pm_1 \eta}{|\eta|} \rangle \right| &\lesssim \left| 1 - \langle \frac{\pm \xi}{|\xi|} , \frac{\pm_1 \eta}{|\eta|} \rangle \right|
+ \frac{|\langle \pm \xi, \pm_1 \eta \rangle|}{|\xi| |\eta| \langle \xi \rangle} 
 \lesssim|1- \cos \angle(\pm \xi, \pm_1 \eta)| + \frac{1}{\langle \xi \rangle} \\
 &\lesssim \angle(\pm \xi,\pm_1 \eta) ^2 + \frac{1}{\langle \xi \rangle}  \lesssim |\angle(\pm \xi,\pm_1 \eta)| + \frac{1}{\langle \xi \rangle} \, .
\end{align*}
The last term is harmless and can be handled as $I_3$ and $I_5$ . The first term
reduces the desired estimate for $I_1$ to
\begin{equation}
\label{9}
\| B_{\pm,\pm_1} (A^{\pm}_0,\psi_{\pm_1}) \|_{X^{l,-\frac{1}{4}+}_{\pm}} \lesssim \|A^{\pm}_0 \|_{X^{s,\frac{3}{4}+}_{\pm}} \|\psi_{\pm_1} \|_{X^{l,\frac{3}{4}+}_{\pm_1}} \, .
\end{equation}
The estimates (\ref{7}),(\ref{8}),(\ref{9}) are implied by Prop. \ref{Prop.1.2}.

 For (\ref{7}) we apply this proposition with parameters $\sigma_0 = s$ , $\beta_0 = \frac{3}{4}--$ , $\sigma_1 = l$ , $\sigma_2=-l$, $\beta_1 = \frac{3}{4}+$ , $\beta_2 = \frac{1}{4}+$ . We require $ s \ge l$ , $\sigma_0+\sigma_1+\sigma_2 + \beta_0 = s+\frac{3}{4}- > 1$, which is fulfilled. Moreover $s > \frac{3}{2} -(s+\frac{3}{4}-l) \, \Leftrightarrow \, 2s-l > \frac{3}{4}$ , which is fulfilled, because $2s-l = s+(s-l) \ge s > \frac{3}{4}$ , and $ s > \frac{3}{2}-(s+l+\frac{1}{4}) \, \Leftrightarrow \, 2s+l > \frac {5}{4}$, which holds for $s > \frac{3}{4} $ , $ l > \frac{1}{4}$ .

For (\ref{8}) and (\ref{9}) we have $\sigma_0=-l$ , $\beta_0 = \frac{1}{4}--$ , $\sigma_1 =l$ , $\beta_1=\frac{3}{4}+$ , $\beta_2 = \frac{3}{4}+$,  so that we require $\sigma_0+\sigma_1+\sigma_2+\beta_0 = s+\frac{1}{4} > 1$ , so that our assumption $s > \frac{3}{4}$ is sharp here. Moreover we require $ s+\frac{1}{4} > \frac{3}{2} -(\frac{1}{4}+s+l) \, \Leftrightarrow \, 2s+l > 1$ , which is satisfied, $s > \frac{3}{2}-(-l+\frac{3}{4}+l) \, \Leftrightarrow \, s > \frac{3}{4}$ and $s> \frac{3}{2}-(-l+s+\frac{3}{4}) \, \Leftrightarrow \, 2s-l > \frac{3}{4}$, which is satisfied as before.
\end{proof}

\section{Proof of Theorem \ref{Theorem1.1}.}
As noticed in Remark 3 it only remains to prove that the solution of Proposition \ref{Prop.1} fulfills
$u := \partial^{\mu} A_{\mu} = 0 $ and $V_{\mu \nu} := F_{\mu \nu} - F[A]_{\mu \nu} = 0 $ . Let us remark that a sketch of the proof for the Yang-Mills equation was given by \cite{ST}.

We recall  $J_{\mu} (\psi): = -\langle \psi_i, \alpha_{\mu} T^a_{ij} \psi_j \rangle T_a$ , so that $J^a_{\mu}(\psi) = -\langle \psi_i, \alpha_{\mu} T^a_{ij} \psi_j \rangle$ and 
$$[A^{\mu},J_{\mu}(\psi)] = [A^b_{\mu} T_b, J^c_{\mu}(\psi) T_c] = A^b_{\mu} J^c_{\mu}(\psi) [T_b,T_c] \, , $$
so that $[A^{\mu},J_{\mu}(\psi)]_a = A^b_{\mu} J^b_{\mu}(\psi) f_{abc} $, where $f_{abc} := [T_b,T_c]_a $ , thus
\begin{equation}
\label{J}
D^{\mu} J^a_{\mu}(\psi) = \partial^{\mu} J^a_{\mu}(\psi) + f_{abc} A^b_{\mu} J^c_{\mu}(\psi) = 0 \, . 
\end{equation}
This is a well-known fact in physics. We refer to \cite{Sz}, formula (25.83), or to \cite{CC}, formula (2.9). We start with the identity
$$D_{\nu} D_{\mu} F[A]^{\mu \nu} = 0 \, , $$
which can be shown by an elementary calculation. We obtain by (\ref{0.4}): 
\begin{align}
\nonumber
D^{\mu} F[A]_{\mu \nu} & = \square A_{\nu} - \partial_{\nu} \partial^{\mu} A_{\mu}+ [\partial^{\mu} A_{\mu},A_{\nu}] + [A^{\mu},\partial^{\mu} A_{\nu}] + [A^{\mu},F[A]_{\mu \nu}] \\
\label{35'}
& =  -\partial_{\nu} u + [u,A_{\nu}] + J_{\nu}(\psi) - [A^{\mu},V_{\mu \nu}] \,.
\end{align}
This implies by (\ref{J}) :
\begin{align*}
0=D_{\nu} D_{\mu} F[A]^{\mu \nu} &= D^{\nu}(-\partial_{\nu} u + [u,A_{\nu}]) + D^{\nu} J_{\nu}(\psi) - D^{\nu} [A^{\mu},V_{\mu \nu}] \\
& = -\partial^{\nu} \partial_{\nu} u + \partial^{\nu}[u,A_{\nu}] - [A_{\nu},\partial^{\nu} u] + [A_{\nu},[u,A^{\nu}]] \\
& \quad - \partial^{\nu}[A^{\mu},V_{\mu \nu}] - [A_{\mu},\partial^{\nu} V_{\mu \nu}] - [A^{\nu},[A^{\mu},V_{\mu \nu}]] \, ,
\end{align*}
so that we obtain
\begin{equation}
\label{*}
\square u = M_1(A_{\nu},\nabla A_{\mu},u,\nabla u,V_{\mu \nu},\nabla V_{\mu \nu}) \, ,
\end{equation}
where $M_j$ here and in what follows denote linear functions in $u,\nabla u,V_{\mu \nu}, \nabla V_{\mu \nu}$ .

In order to prove a similar equation for $\square V_{\mu \nu}$ we start with the Bianchi identity
$$D_{\mu} F[A]_{\beta \gamma} + D_{\beta} F[A]_{\gamma \mu} + D_{\gamma} F[A]_{\mu \beta} = 0 \, . $$
This implies
$$ - D^{\mu} D_{\mu} V_{\beta \gamma} + D^{\mu}D_{\mu} F_{\beta \gamma} + D^{\mu} D_{\beta} F[A]_{\gamma \mu} + D^{\mu} D_{\gamma} F[A]_{\mu \beta} = 0 \, . $$
Using the commutation identity
$$ D^{\mu}D_{\beta} F[A]_{\gamma \mu} =  D^{\beta}D_{\mu} F[A]_{\gamma \mu} + [F[A]_{\mu \beta},F[A]_{\gamma \mu}] $$
we obtain
\begin{align}
\nonumber
&- D^{\mu}D_{\mu} V_{\beta \gamma} + D^{\mu} D^{\mu} F_{\beta \gamma} +  D^{\beta}D^{\mu} F[A]_{\gamma \mu} + [F[A]^{\mu \beta},F[A]_{\gamma \mu}] \\
\label{36'}
&  \hspace{10em}
+  D^{\gamma}D^{\mu} F[A]_{\mu \beta} + [F[A]^{\mu \gamma},F[A]_{\mu \beta}] = 0 \, .
\end{align}
(\ref{35'}) implies
\begin{align*}
D^{\beta}D^{\mu} F[A]_{\gamma \mu} + D^{\gamma}D^{\mu} F[A]_{\mu \beta} & = D^{\beta}(\partial_{\gamma}u -[u,A_{\gamma}] +[A^{\mu},V_{\mu \gamma}]-J_{\gamma}(\psi)) \\
& \quad + D^{\gamma}(-\partial_{\beta}u+[u,A_{\beta}] -[A^{\mu},V_{\mu \beta}]+J_{\beta}(\psi)) \, ,
\end{align*}
so that (\ref{36'}) implies
\begin{align*}
&-D^{\mu} D_{\mu}V_{\beta \gamma} + D^{\mu} D_{\mu} F_{\beta \gamma} + \partial^{\beta}\partial_{\gamma} u - \partial_{\beta}[u,A_{\gamma}]+\partial_{\beta}[A^{\mu},V_{\mu \gamma}] - D^{\beta} J_{\gamma}(\psi) \\& - \partial^{\gamma} \partial_{\beta}u + \partial^{\gamma}[u,A_{\beta}]-\partial^{\gamma}[A^{\mu},V_{\mu \beta}] + D^{\gamma} J_{\beta}(\psi) \\
&+[A_{\beta},\partial_{\beta}[u,A_{\gamma}]] + [A_{\beta},\partial^{\beta}[A^{\mu},V_{\mu \gamma}]]
-[A_{\gamma},\partial^{\gamma}[u,A_{\beta}]] - [A_{\gamma},\partial^{\gamma}[A^{\mu},V_{\mu \beta}]] \\
& + [F[A]^{\mu \beta},F[A]_{\gamma \mu}] + [F[A]^{\mu \gamma},F[A]_{\mu \beta}]=0 \, ,
\end{align*}
which means that
\begin{align*}
-D^{\mu} D_{\mu} V_{\beta \gamma} + D^{\mu} D_{\mu} F_{\beta \gamma} +& M_2(A_{\nu},\nabla A_{\nu},u,\nabla u,V_{\mu \nu},\nabla V_{\mu \nu})\\ &+ 2[F[A]^{\mu \beta},F[A]_{\gamma \mu} ]- D^{\beta} J_{\gamma}(\psi) + D^{\gamma} J_{\beta}(\psi) = 0 \, .
\end{align*}
By (\ref{*0}) this implies
\begin{equation}
\label{**}
D^{\mu} D_{\mu} V_{\beta \gamma} = 2([F^{\mu \beta},F_{\gamma \mu}]-[F[A]^{\mu \beta},F[A]_{\gamma \mu}]) - M_2(A_{\nu},\nabla A_{\nu},u,\nabla u,V_{\mu \nu}, \nabla V_{\mu \nu}) \, .
\end{equation}
Furthermore we obtain
\begin{align*}
D^{\mu} D_{\mu} V_{\beta \gamma} &= \partial^{\mu}(D_{\mu} V_{\beta \gamma}) + [A^{\mu},D_{\mu} V_{\beta \gamma}] = \square V_{\beta \gamma} +  M_3(A_{\nu},\nabla A_{\nu},V_{\beta \gamma}, \nabla V_{\beta \gamma}) 
\end{align*} 
and
\begin{align*}
&[F^{\mu \beta},F_{\gamma \mu}] - [F[A]^{\mu \beta},F[A]_{\gamma \mu}] = [F[A]^{\mu \beta}+V^{\mu \beta},F[A]_{\gamma \mu}+V_{\gamma \mu}] -[F[A]^{\mu \beta},F[A]_{\gamma \mu}] \\
& = [V^{\mu \beta},F[A]_{\gamma \mu}+V_{\gamma \mu}]
+ [F[A]^{\mu \beta},V_{\gamma \mu}] \\
& = [V^{\mu \beta},F[A]_{\gamma \mu}] + [V^{\mu \beta},F_{\gamma \mu} - F[A]_{\gamma \mu}] + [F[A]^{\mu \beta},V_{\gamma \mu}] = M_4(A_{\nu},\nabla A_{\nu},F_{\gamma \mu})V_{\mu \beta}  
\end{align*}
Thus (\ref{**}) reduces to
\begin{equation}
\label{***}\square V_{\beta \gamma} = M_5(A_{\nu},\nabla A_{\nu},F_{\gamma \mu},u,\nabla u, V_{\mu \nu}, \nabla V_{\mu\nu})
\end{equation}
We have shown by (\ref{*}) and (\ref{***}) that $u$ and $V_{\mu \nu}$ fulfill a system of linear wave equations. Our aim is to show that $u=V_{\mu\nu} =0$ . It remains to prove under our assumptions on the data  $$u(0) = (\partial_t u)(0) = V_{\mu \nu}(0) = (\partial_t V_{\mu \nu})(0) = 0 \,.$$
Obviously we obtain by (\ref{Const}) :
$$u(0) = -(\partial_t A_0)(0) + (\partial_j A^j)(0) = -a_1^0 + \partial_j a_0^j =0 \, . $$
Next we calculate
\begin{align*}
\partial_t u & = -\partial_t(\partial_t A_0 - \partial_j A^j) = \square A_0 - \Delta A_0 + \partial_t \partial_j A^j \\
& = -[A^{\mu},\partial_{\mu}A_0] - [A^{\mu},F_{\mu \nu}] - \Delta A_0 + \partial_t \partial_j A^j + J_0(\psi) \\
& =- [A^k,\partial_k A_0] -[A^k,F_{k0}] + [A_0,\partial_t A_0]- \Delta A_0 + \partial_j(\partial_t A^j) + J_0(\psi) \, ,
\end{align*}
which implies by (\ref{f2}) :
\begin{align*}
&(\partial_t u)(0) = -[a^k_0,\partial_k a^0_0] + [a^0_0,a_1^0] - [a^k_0,f^{k0}_0] - \partial^j \partial_j a^0_0 + \partial_j a^j_1 + J_0(\psi)(0) \\
& = -[a^k_0,\partial_k a^0_0] + [a^0_0,a_1^0] - [a^k_0,f^{k0}_0] - \partial_j(-f^{0j}_0 + a_1^j+[a^0_0,a^j_0]) + \partial_j a^j_1 + J_0(\psi)(0) \\
& = -[a^k_0,\partial_k a^0_0] + [a^0_0,a_1^0] - [a^k_0,f^{k0}_0] + \partial_j f^{0j}_0 - \partial_j[a^0_0,a^j_0] + J_0(\psi)(0) \, .
\end{align*}
By (\ref{Const}) we obtain
$$ \partial_j f^{0j}_0 = -\partial_j f_0^{j0} = [a_0^k,f^{k0}_0] - J_0(\psi)(0) \, , $$
so that we arrive at
$$ (\partial_t u)(0) =  -[a^k_0,\partial_k a^0_0] + [a^0_0,a_1^0] - [\partial_j a^0_0,a^j_0] -[a^0_0,\partial_j a^j_0] = 0 \, ,
$$ where we used (\ref{Const}) in the last line.

Next we obtain by (\ref{f1}) and (\ref{f2}) :
$$ V_{\mu \nu}(0) = f^{\mu \nu}_0 - (\partial_{\mu} A_{\nu}-\partial_{\nu} A_{\mu} + [A_{\mu},A_{\nu}])(0) = 0 $$
and
\begin{align*}
&(\partial_t V_{ij})(0) = (\partial_t F_{ij})(0) - (\partial_t(\partial_i A_j - \partial_j A_i + [A_i,A_j]))(0) \\
& \quad = \partial_i a^j_1 - \partial_j a^i_1 + [a^i_1,a^j_0] + [a^i_0,a^j_1] - ( \partial_i a^j_1 - \partial_j a^i_1 + [a^i_1,a^j_0] + [a^i_0,a^j_1]) = 0 \, .
\end{align*}
Moreover
\begin{align*}
(\partial_t V_{0i})(0) & = (\partial_t F^{0i})(0) -(\partial_t(\partial_t A_i - \partial_i A_0 +[A_0,A_i])(0)) \\
& = f^{0i}_1 - (\partial_t^2 A_i)(0) +(\partial_t \partial_i A_0)(0) -[a^0_1,a^i_0]-[a^0_i,a^i_1] \\
& = f^{0i}_1 - [A^{\mu},\partial^{\mu}A_i](0) - [A^{\mu},F_{\mu i}](0) + J_i(\psi)(0) - \partial_j(\partial^j a^i_0) \\
& \quad + \partial_i a^0_1 -[a^0_1,a^i_0] -[a^0_0,a^i_1] \, ,
\end{align*}
where we used $-\partial^2_t A_i = \square A_i - \Delta A_i$ and (\ref{0.4}). Using (\ref{f4}) we obtain
\begin{align*}
(\partial_t V_{0i})(0) &  = \partial_k f^{ki}_0 - [a^0_0,f^{ki}_0] + [a^k_0,f^{ki}_0] - J_i(\psi)(0) + [a^0_0,a^i_1] -[a^k_0,\partial_k a^i_0] \\& \,\,\, + [a^0_0,f^{0i}_0] - [a^k_0,f^{ki}_0] + J_i(\psi)(0) - \partial^j(\partial_j a^i_0) + \partial_ia^0_1 -[a^0_1,a^i_0] -[a^0_0,a^i_1] \, .
\end{align*}
Because by (\ref{f1}) 
\begin{align*}
\partial_k f^{ki}_0 & = -\partial_k f^{ik}_0 = -\partial_k(\partial_i a^k_0 - \partial^k a^i_0) - \partial_k[a^i_0,a^k_0] \\& = -\partial_i a^0_1 + \partial^k \partial_k a^i_0 - [\partial_k a^i_0,a^k_0] - [a^i_0,a^0_1]
\end{align*}
we end up with
$$ (\partial_t V_{0i})(0) = 0 \,. $$

\end{document}